\documentclass{svproc}
\usepackage{amsmath, verbatim, amsfonts, amssymb, color}
\usepackage{mathrsfs}
\usepackage{dsfont}

\newtheorem{thm}{Theorem}
\newtheorem{prop}[thm]{Proposition}

\newtheorem{exam}[thm]{Example}
\newtheorem{rem}[thm]{Remark}
\newtheorem{lem}[thm]{Lemma}
\newtheorem{defn}[thm]{Definition}
\newtheorem{assum}[thm]{Assumption}

\newcommand{\supp}{\mathop{\mathrm{supp}}}

\renewcommand\labelenumi{\textup{\alph{enumi})}}
\renewcommand\theenumi\labelenumi

 \makeatletter
\def\@makefnmark{\hbox{(\@textsuperscript{\normalfont\@thefnmark})}}
\makeatother

\makeatletter
\@namedef{subjclassname@2020}{%
  \textup{2020} Mathematics Subject Classification}
\makeatother

\providecommand{\ack}[1]{\par\addvspace\baselineskip
\noindent\ackname\enspace\ignorespaces#1}%
\def\subjclassname{\textup{2020} \textit{Mathematics Subject Classification:}}%
\providecommand{\subjclass}[1]{\par\addvspace\baselineskip
\noindent\subjclassname\enspace\ignorespaces#1}%

\begin{document}
\mainmatter
\title{\bfseries Maximal displacement of branching symmetric stable processes}
\titlerunning{Maximal displacement of branching symmetric stable processes}

\author{Yuichi Shiozawa}
\authorrunning{Yuichi Shiozawa}

\institute{Department of Mathematics, Graduate School of Science, Osaka University\\
\email{shiozawa@math.sci.osaka-u.ac.jp}  }

\maketitle

\begin{abstract}  
We determine the limiting distribution and the explicit tail behavior for 
the maximal displacement of a branching symmetric stable process 
with spatially inhomogeneous branching structure. 
Here the branching rate is a Kato class measure with compact support and can be singular 
with respect to the Lebesgue measure.

\end{abstract}

\section{Introduction}
We studied in \cite{NS21+} the limiting distributions 
for the maximal displacement of a branching Brownian motion 
with spatially inhomogeneous branching structure. 
In this paper, we show the corresponding results 
for a branching symmetric stable process. 
Our results clarify how the tail behavior of the underlying process 
would affect the long time asymptotic properties of 
the maximal displacement.

There has been great interest in the maximal displacement of branching Brownian motions 
and branching random walks with light tails 
for which the associated branching structures are spatially homogeneous.   
We refer to \cite{B78} as a pioneering work, 
and to \cite[Sections 5--7]{B17} and references therein, and  \cite{KLZ21+,SBM21+} for recent developments 
on this research subject. 
On the other hand, Durrett \cite{D83} proved the weak convergence of 
a properly normalized maximal displacement   
of a branching random walk on ${\mathbb R}$ with regularly varying tails. 
This result in particular shows that 
the maximal displacement grows exponentially in contrast with the light tailed model. 
Bhattacharya-Hazra-Roy \cite[Theorem 2.1]{BHR17} further showed 
the weak convergence of point processes associated with the scaled particle positions. 
We refer to \cite{LS16,GMV17} for related studies 
on the maximal displacement of branching stable processes or branching random walks 
with regularly varying tails.  

Here we are interested in how the spatial inhomogeneity of the branching structure 
would affect the behavior of the maximal displacement. 
For a branching Brownian motion on ${\mathbb R}$,  
Erickson \cite{E84} determined the linear growth rate of the maximal displacement 
in terms of the principal eigenvalue of the Schr\"odinger type operator. 
Lalley-Sellke \cite{LS88} then showed that 
a properly shifted maximal displacement 
is weakly convergent to a random shift of the Gumbel distribution with respect to the limiting martingale.
Bocharov and Harris \cite{BH14,BH16} proved the results corresponding to \cite{E84,LS88}  
for a catalytic branching Brownian motion on ${\mathbb R}$ 
in which reproduction occurs only at the origin. 
We showed in \cite{S18,S19} that the results in \cite{E84,BH14} are valid 
for the branching Brownian motion motion on ${\mathbb R}^d \ (d\geq 1)$
in which the branching rate is a Kato class measure with compact support. 
Under the same setting, Nishimori and the author \cite{NS21+} further 
determined the limiting distribution of the shifted maximal displacement as in \cite{LS88,BH16}. 
This result reveals that, even though reproduction occurs only on a compact set unlike \cite{LS88}, 
the spatial dimension $d$ appears in the lower order term of the maximal displacement. 
We refer to \cite{Bo20,N21+} for further developments.

Carmona-Hu \cite{CH14} and Bulinskaya \cite{B18,B20} 
obtained the linear growth rate and weak convergence 
for the maximal displacement of a branching random walk on ${\mathbb Z}^d$ 
in which reproduction occurs only on finite points 
and the underlying random walk has light tails. 
We note that the underlying random walk in \cite{B18,B20,CH14} is irreducible and allowed to be nonsymmetric, 
and the so-called $L\log L$ condition is sufficient 
for the validity of these results as proved by \cite{B18,B20}. 
Recently, Bulinskaya \cite{B21} showed 
the weak convergence of a properly normalized maximal displacement 
of a branching random walk on ${\mathbb Z}$ with 
 regularly varying tails as in \cite{D83,BHR17} 
and reproduction occurs only on finite points. 
As for the spatially homogeneous model, 
the growth rate of the maximal displacement is exponential in contrast with the light tailed model.

In this paper, we prove the weak convergence and the long time tail behavior 
for the maximal displacement of a branching symmetric stable process on ${\mathbb R}^d$
with spatially inhomogeneous branching structure (Theorems \ref{thm:weak} and \ref{thm:tail}). 
We will then see that the maximal displacement grows exponentially 
and the growth rate is determined 
by the principal eigenvalue of the Schr\"odinger type operator 
associated with the fractional Laplacian. 
The spatial dimension $d$ also affects the limiting distribution and tail behavior of the maximal displacement. 
Our results are applicable to a branching symmetric stable process 
in which reproduction occurs only on singular sets.

Our results can be regarded as a continuous state space and multidimensional analogue of \cite{B21}.
In particular, we provide an explicit form of the limiting distribution for the maximal displacement. 
On the other hand, since our approach is based on the second moment method as for \cite{NS21+},  
we need the second moment condition on the offspring distribution,  
which is stronger than the $L\log L$ condition as imposed in \cite{BHR17,B21,D83}.

We note that the functional analytic approach works well for the continuous space model. 
Our model of branching symmetric stable processes is closely related to  
the Schr\"odinger type operator 
associated with the fractional Laplacian and Kato class measure 
through the first moment formula on the expected population 
(Lemma \ref{lem:moment}). 
It is  possible to calculate or estimate 
the principal eigenvalue of the  Schr\"odinger type operator 
as in Subsection \ref{subsect:exam}.

The rest of this paper is organized as follows: 
In Section \ref{sect:stable}, 
we first prove the resolvent asymptotic behaviors of a symmetric stable process. 
We then discuss the invariance of the essential spectra,  
and the asymptotics of the integral associated with the ground state,  
for the Schr\"odinger type operator. 
We finally determine asymptotic behaviors of the Feynman-Kac functional. 
In Section \ref{sect:branching}, we first introduce a model of branching symmetric stable processes. 
We then present our results in this paper with examples.  
In Section \ref{sect:proof}, we prove the weak convergence result 
(Theorem \ref{thm:weak}) by following the approach of \cite[Theorem 2.4]{NS21+}.

\section{Symmetric stable processes}\label{sect:stable}
For $\alpha\in (0,2)$, 
let ${\mathbf M}=(\{X_t\}_{t\geq 0}, \{P_x\}_{x\in {\mathbb R}^d})$ 
be a symmetric $\alpha$-stable process on ${\mathbb R}^d$ 
generated by $-(-\Delta)^{\alpha/2}/2.$
Let $p_t(x,y)$ be the transition density function of ${\mathbf M}$, 
$$P_x(X_t\in A)=\int_A p_t(x,y)\,{\rm d}y, \quad t>0, x\in {\mathbb R}^d, A\in {\cal B}({\mathbb R}^d).$$
Here ${\cal B}({\mathbb R}^d)$ is the family of Borel measurable sets on ${\mathbb R}^d$.  
By \cite[Theorem 2.1]{BG60} and \cite[Lemmas 2.1 and 2.2]{W17}, we have
\begin{lem}
There exists a positive continuous function $g$ on $[0,\infty)$ such that 
\begin{equation}\label{eq:scaling}
p_t(x,y)=\frac{1}{t^{d/\alpha}}g\left(\frac{|x-y|}{t^{1/\alpha}}\right).
\end{equation}
Moreover, the function $g$ satisfies the following.
\begin{enumerate}
\item[{\rm (i)}] The value $g(0)$ is given by 
$$g(0)=\frac{2^{d/\alpha}\Gamma(d/\alpha)}{\alpha2^{d-1}\pi^{d/2}\Gamma(d/2)}.$$
\item[{\rm (ii)}] 
There exists $c>0$ such that for any $r\geq 0$,
\begin{equation}\label{eq:w-asymp}
0\leq g(0)-g(r)\leq cr^2. 
\end{equation}
\item[{\rm (iii)}] 
The next equality holds.
\begin{equation}\label{eq:bg-asymp}
\lim_{r\rightarrow\infty}r^{d+\alpha}g(r)
=\frac{\alpha\sin(\alpha\pi/2)\Gamma((d+\alpha)/2)\Gamma(\alpha/2)}{2^{2-\alpha}\pi^{1+d/2}}(=:C_{d,\alpha}).
\end{equation}
\end{enumerate}
\end{lem}

\subsection{Resolvent asymptotics}

In this subsection, we prove the asymptotic behaviors 
and the so-called $3$G-type inequality for the resolvent density of ${\mathbf M}$. 
For $\beta>0$, let $G_{\beta}(x,y)$ denote the $\beta$-resolvent density 
of ${\mathbf M}$,
$$G_{\beta}(x,y)=\int_0^{\infty}e^{-\beta t}p_t(x,y)\,{\rm d}t.$$ 
Define 
$$w_{\beta}(r)
=\int_0^{\infty}e^{-\beta t}\frac{1}{t^{d/\alpha}}g\left(\frac{r}{t^{1/\alpha}}\right)\,{\rm d}t \ (r\geq 0)$$
so that $G_{\beta}(x,y)=w_{\beta}(|x-y|)$ by \eqref{eq:scaling}.

\begin{lem}\label{lem:resol} Let $\beta>0$.
\begin{enumerate}
\item[{\rm (i)}] The function $w_{\beta}$ satisfies 
\begin{equation}\label{eq:res-inf}
\lim_{r\rightarrow\infty}r^{d+\alpha}w_{\beta}(r)
=\beta^{-2}C_{d,\alpha}
\end{equation}
and
\begin{equation}\label{eq:res-0}
w_{\beta}(r)\sim 
\begin{cases}
\beta^{(d-\alpha)/\alpha}\Gamma((\alpha-d)/d)g(0) & (d<\alpha), \\ 
\alpha g(0) \log r^{-1} & (d=\alpha),\\
\alpha r^{\alpha-d}\int_0^{\infty}s^{d-\alpha-1}g(s)\,{\rm d}s & (d>\alpha),
\end{cases}
\quad (r\rightarrow +0).
\end{equation}
\item[{\rm (ii)}] 
There exists $C>0$ such that for any $x,y,z\in {\mathbb R}^d$,
$$G_{\beta}(x,y)G_{\beta}(y,z)\leq CG_{\beta}(x,z)\left(G_{\beta}(x,y)+G_{\beta}(y,z)\right).$$
\end{enumerate}
\end{lem}

\begin{proof}
(ii) follows by (i) and direct calculation. 
We now show (i). 
The relation \eqref{eq:res-inf} follows by \eqref{eq:bg-asymp} 
and the dominated convergence theorem:
\begin{equation*}
\begin{split}
r^{d+\alpha}w_{\beta}(r)
&=\int_0^{\infty}e^{-\beta t}\left(\frac{r}{t^{1/\alpha}}\right)^{d+\alpha}
g\left(\frac{r}{t^{1/\alpha}}\right)t\,{\rm d}t
\rightarrow \beta^{-2}C_{d,\alpha}\quad (r\rightarrow\infty).
\end{split}
\end{equation*}

We next show \eqref{eq:res-0}. If $d<\alpha$, then as $r\rightarrow +0$,
$$
w_{\beta}(r)
\rightarrow 
\int_0^{\infty}e^{-\beta t}t^{-d/\alpha}\,{\rm d}t \, g(0)
=\beta^{(d-\alpha)/\alpha}\Gamma\left(\frac{\alpha-d}{\alpha}\right)g(0).
$$
If $d>\alpha$, then \eqref{eq:bg-asymp} implies that as $r\rightarrow +0$, 
\begin{equation*}
r^{d-\alpha}w_{\beta}(r)
=\alpha \int_0^{\infty}e^{-\beta (r/s)^{\alpha}}
s^{d-\alpha-1} g(s)\,{\rm d}s
\rightarrow \alpha \int_0^{\infty}
s^{d-\alpha-1} g(s)\,{\rm d}s. 
\end{equation*}

We now assume that $d=\alpha(=1)$. 
Let 
$$
w_{\beta}(r)
=\int_0^{r^{\alpha}}e^{-\beta t}t^{-1}g\left(\frac{r}{t^{1/\alpha}}\right)\,{\rm d}t
+\int_{r^{\alpha}}^{\infty}e^{-\beta t}t^{-1}g\left(\frac{r}{t^{1/\alpha}}\right)\,{\rm d}t
={\rm (I)}+{\rm (II)}.
$$
Then by \eqref{eq:bg-asymp}, 
$${\rm (I)}\leq \frac{c_1}{r^{2\alpha}}\int_0^{r^{\alpha}}e^{-\beta t}t\,{\rm d}t\leq c_2.$$
Let 
$$
{\rm (II)}
=\int_{r^{\alpha}}^{\infty}e^{-\beta t}t^{-1}
\left(g\left(\frac{r}{t^{1/\alpha}}\right)-g(0)\right)\,{\rm d}t
+g(0)\int_{r^{\alpha}}^{\infty}e^{-\beta t}t^{-1}\,{\rm d}t.
$$
Then by \eqref{eq:w-asymp}, 
$$\left|\int_{r^{\alpha}}^{\infty}e^{-\beta t}t^{-1}
\left(g\left(\frac{r}{t^{1/\alpha}}\right)-g(0)\right)\,{\rm d}t\right|
\leq c_3r^2\int_{r^{\alpha}}^{\infty}e^{-\beta t}t^{-1-2/\alpha}\,{\rm d}t\leq c_4.$$
Since 
$$\int_{r^{\alpha}}^{\infty}e^{-\beta t}t^{-1}\,{\rm d}t
\sim \alpha \log\left(\frac{1}{r}\right) \quad (r\rightarrow +0),$$
we obtain the desired assertion for $d=\alpha$.
\qed
\end{proof}

If $d>\alpha$, then ${\mathbf M}$ is transient 
and the Green function $G(x,y)=\int_0^{\infty}p_t(x,y)\,{\rm d}t$ is given by 
\begin{equation}\label{eq:green}
G(x,y)=\frac{2^{1-\alpha}\Gamma((d-\alpha)/2)}{\pi^{d/2}\Gamma(\alpha/2)}|x-y|^{\alpha-d}.
\end{equation}
We use the notation $G_0(x,y)=G(x,y)$ for $d>\alpha$. 

\subsection{Spectral properties of Schr\"odinger type operators with the fractional Laplacian}
\label{subsect:spectral}
In this subsection, we study spectral properties of a Schr\"odinger type operator 
associated with the fractional Laplacian and Green tight Kato class measure. 
We first introduce Kato class and Green tight measures, and the associated bilinear forms.   
\begin{defn}
\begin{enumerate}
\item[{\rm (i)}] Let $\mu$ be a positive Radon measure on ${\mathbb R}^d$. 
Then $\mu$ belongs to the Kato class {\rm (}$\mu\in {\cal K}$ in notation{\rm )} 
if 
$$\lim_{\beta\rightarrow\infty}
\sup_{x\in {\mathbb R}^d}\int_{{\mathbb R}^d}G_{\beta}(x,y)\,\mu({\rm d}y)=0.$$
\item[{\rm (ii)}] A measure $\mu\in {\cal K}$ is  {\rm (}$1$-{\rm )}Green tight 
{\rm (}$\mu\in {\cal K}_{\infty}(1)$ in notation{\rm )} if 
$$\lim_{R\rightarrow\infty}
\sup_{x\in {\mathbb R}^d}\int_{|y|>R}G_1(x,y)\,\mu({\rm d}y)=0.$$
\end{enumerate}
\end{defn}

Any Kato class measure with compact support  in ${\mathbb R}^d$ 
belongs to ${\cal K}_{\infty}(1)$ by definition.

Let $({\cal E},{\cal F})$ be a regular Dirichlet form on $L^2({\mathbb R}^d)$ 
associated with ${\mathbf M}$, 
\begin{equation*}
\begin{split}
{\cal F}&=\left\{u\in L^2({\mathbb R}^d) \mid \iint_{{\mathbb R}^d\times {\mathbb R}^d}
\frac{(u(x)-u(y))^2}{|x-y|^{d+\alpha}}\,{\rm d}x{\rm d}y<\infty\right\},\\
{\cal E}(u,u)&=\frac{1}{2}{\cal A}(d,\alpha)\iint_{{\mathbb R}^d\times {\mathbb R}^d}
\frac{(u(x)-u(y))^2}{|x-y|^{d+\alpha}}\,{\rm d}x{\rm d}y, \ u\in{\cal F} 
\end{split}
\end{equation*}
with 
$${\cal A}(d,\alpha)=\frac{\alpha2^{d-1}\Gamma((d+\alpha)/2)}{\pi^{d/2}\Gamma(1-\alpha/2)}$$
(\cite[Example 1.4.1]{FOT11}). 
For $\beta>0$, we let ${\cal E}_{\beta}(u,u)={\cal E}(u,u)+\beta\int_{{\mathbb R}^d}u^2\,{\rm d}x$.
Since any function in ${\cal F}$ admits a quasi continuous version (\cite[Theorem 2.1.3]{FOT11}), 
we may and do assume that if we write $u\in {\cal F}$, 
then $u$ denotes its quasi continuous version.  

For $\mu\in {\cal K}$ and $\beta>0$, 
let $G_{\beta}\mu(x)=\int_{{\mathbb R}^d}G_{\beta}(x,y)\,\mu({\rm d}y)$. 
Then $\|G_{\beta}\mu\|_{\infty}<\infty$ by the definition of the Kato class measure. 
Moreover, the Stollmann-Voigt inequality (\cite[Theorem 3.1]{SV96} and \cite[Exercise 6.4.4]{FOT11}) 
holds: 
\begin{equation}\label{eq:sv}
\int_{{\mathbb R}^d}u^2\,{\rm d}\mu\leq \|G_{\beta}\mu\|_{\infty}{\cal E}_{\beta}(u,u), \ u\in {\cal F}.
\end{equation}
In particular, any function $u\in {\cal F}$ belongs to $L^2(\mu)$.
We also know by \eqref{eq:sv} that the embedding 
$$I_{\mu}: ({\cal F},\sqrt{{\cal E}_{\beta}}) \to L^2(\mu), \quad I_{\mu}f=f, \ \text{$\mu$-a.e.}$$
is continuous. 
Even if $\mu$ is a signed measure, we can define the continuous embedding $I_{\mu}$ as above
by replacing $\mu$ with $|\mu|$.

Let $\nu$ be a signed Borel measure on ${\mathbb R}^d$ such that
the measures $\nu^+$ and $\nu^{-}$ in the Jordan decomposition $\nu=\nu^+-\nu^{-}$ belong to ${\cal K}$.  
Let $({\cal E}^{\nu},{\cal F})$ be the quadratic form on $L^2({\mathbb R}^d)$ defined by 
$${\cal E}^{\nu}(u,u)={\cal E}(u,u)-\int_{{\mathbb R}^d}u^2\,{\rm d}\nu, \ u\in {\cal F}.$$
Since $\nu$ charges no set of zero capacity (\cite[Theorem 3.3]{ABM91}), 
$({\cal E}^{\nu},{\cal F})$ is well-defined. 
Furthermore, it is 
a lower bounded closed symmetric bilinear form on $L^2({\mathbb R}^d)$ (\cite[Theorem 4.1]{ABM91}) 
so that the associated self-adjoint operator ${\cal H}^{\nu}$ on $L^2({\mathbb R}^d)$ 
is formally written as ${\cal H}^{\nu}=(-\Delta)^{\alpha/2}/2-\nu$.

Let $\{p_t^{\nu}\}_{t>0}$ be the strongly continuous symmetric semigroup on $L^2({\mathbb R}^d)$ 
generated by $({\cal E}^{\nu},{\cal F})$. 
Let $A_t^{\nu^{+}}$ and $A_t^{\nu^{-}}$ be the positive continuous additive functionals 
in the Revuz correspondence to $\nu^+$ and $\nu^{-}$, respectively 
(see \cite[p.401]{FOT11} for details).  
If we define $A_t^{\nu}=A_t^{\nu^{+}}-A_t^{\nu^{-}}$, then 
$$p_t^{\nu}f(x)=E_x\left[e^{A_t^{\nu}}f(X_t)\right], 
\quad f\in L^2({\mathbb R}^d)\cap {\cal B}_b({\mathbb R}^d).$$ 
Here ${\cal B}_b({\mathbb R}^d)$ is 
the family of bounded Borel measurable functions on ${\mathbb R}^d$. 
Moreover, there exists a jointly continuous integral kernel $p_t^{\nu}(x,y)$ 
on $(0,\infty)\times {\mathbb R}^d\times{\mathbb R}^d$ such that 
$p_t^{\nu}f(x)=\int_{{\mathbb R}^d}p_t^{\nu}(x,y)f(y)\,{\rm d}y$ (\cite[Theorem 7.1]{ABM91}). 

For $\beta\geq 0$, we define 
$G_{\beta}^{\nu}f(x)=\int_0^{\infty}e^{-\beta t}p_t^{\nu}f(x)\,{\rm d}t$ 
provided that the right hand side above makes sense. 
Then 
$$G_{\beta}^{\nu}f(x)=E_x\left[\int_0^{\infty}e^{-\beta t+A_t^{\nu}}f(X_t)\,{\rm d}t\right].$$
If we define $G_{\beta}^{\nu}(x,y)=\int_0^{\infty}e^{-\beta t}p_t^{\nu}(x,y)\,{\rm d}t$, 
then
$G_{\beta}^{\nu}f(x)=\int_{{\mathbb R}^d}G_{\beta}^{\nu}(x,y)f(y)\,{\rm d}y$.

We next discuss the invariance of the essential spectra of $(-\Delta)^{\alpha/2}/2$ 
under the perturbation with respect to the finite Kato class measure. 
For a self-adjoint operator ${\cal L}$ on $L^2({\mathbb R}^d)$, 
let $\sigma_{\rm ess}({\cal L})$ denote its essential spectrum.  
\begin{prop}\label{prop:ess}
If $\nu^{+}$ and $\nu^{-}$ are finite Kato class measures, 
then $\sigma_{\rm ess}({\cal H}^{\nu})=\sigma_{\rm ess}((-\Delta)^{\alpha/2}/2)=[0,\infty)$.
\end{prop}

We refer to \cite{Be04} and \cite{BEKS94} for $\alpha=2$. 
To prove Proposition \ref{prop:ess}, we follow the argument of \cite[Theorem 3.1]{BEKS94}. 
More precisely, we show three lemmas below for the proof of Proposition \ref{prop:ess}. 

Suppose that $\nu$ is a signed Borel measure on ${\mathbb R}^d$ 
such that $\nu^+$ and $\nu^{-}$ are positive Radon measures on ${\mathbb R}^d$ 
charging no set of zero capacity. 
For $\beta>0$, let $G_{\beta}\nu(x)=\int_{{\mathbb R}^d}G_{\beta}(x,y)\,\nu({\rm d}y)$.
Then for any $f\in L^2(|\nu|)$, 
$(f\nu)^+$ and $(f\nu)^{-}$ are positive Radon measures on ${\mathbb R}^d$ 
charging no set of zero capacity and
$$G_{\beta}(f\nu)(x)
=E_x\left[\int_0^{\infty}e^{-\beta t}f(X_t)\,{\rm d}A_t^{\nu}\right]
=E_x^{\beta}\left[A_{\zeta}^{f\nu}\right].$$
Here $P_x^{\beta}$ is the law of the $e^{-\beta t}$-subprocess of ${\mathbf M}$ 
and $\zeta$ is the lifetime of this subprocess. 
By \cite[Theorem 6.7.4]{CF12}, we have $G_{\beta}(f\nu)\in {\cal F}$ and
${\cal E}_{\beta}(G_{\beta}(f\nu),v)
=\int_{{\mathbb R}^d}f\cdot I_{\nu}v\,{\rm d}\nu$ for any $v\in {\cal F}$.
Hence if we define 
$K_{\beta}f(x)=G_{\beta}(f\nu)(x)$ for $f\in L^2(|\nu|)$, 
then $K_{\beta}f\in {\cal F}$.

\begin{lem}\label{lem:bdd}
If $\nu^{+}$ and $\nu^{-}$ belong to the Kato class, 
then for any $\beta>0$, $I_{\nu}K_{\beta}$ is a bounded linear operator on $L^2(|\nu|)$. 
Moreover, there exists $\beta_0>0$ such that for any $\beta>\beta_0$, 
the associated operator norm  $\|I_{\nu}K_{\beta}\|$ satisfies $\|I_{\nu}K_{\beta}\|<1$.  
\end{lem}

\begin{proof} \rm 
We prove this lemma only for $\nu^{-}=0$   
because a similar calculation applies to the general case.  
By \eqref{eq:sv}, 
we have for any $f\in L^2(\nu)$,
$$
\int_{{\mathbb R}^d}(I_{\nu}K_{\beta}f)^2\,{\rm d}\nu
=\int_{{\mathbb R}^d}(K_{\beta}f)^2\,{\rm d}\nu
\leq \|G_{\beta}\nu\|_{\infty}{\cal E}_{\beta}(K_{\beta}f,K_{\beta}f)<\infty
$$
so that $I_{\nu}K_{\beta}f\in L^2(\nu)$. 
Combining this with the relation
$$
{\cal E}_{\beta}(K_{\beta}f,K_{\beta}f)
=\int_{{\mathbb R}^d}(I_{\nu}K_{\beta}f) f\,{\rm d}\nu
\leq \sqrt{\int_{{\mathbb R}^d}(I_{\nu}K_{\beta}f)^2\,{\rm d}\nu}\sqrt{\int_{{\mathbb R}^d}f^2\,{\rm d}\nu},
$$
we get 
$\|I_{\nu}K_{\beta}\|\leq \|G_{\beta}\nu\|_{\infty}$. 
Since  $\nu$ is a Kato class measure , we have $\|G_{\beta}\nu\|_{\infty}\rightarrow 0$ 
as $\beta\rightarrow\infty$ so that the desired assertion holds.
\qed
\end{proof}

Lemma \ref{lem:bdd} implies that for any $\beta>\beta_0$, we can define $(1-I_{\nu}K_{\beta})^{-1}$ 
as a bounded linear operator on $L^2(|\nu|)$. 

\begin{lem}\label{lem:resol-1}
Let $\beta_0$ and $\nu^{\pm}$ be as in Lemma {\rm \ref{lem:bdd}}. 
Then for any $\beta>\beta_0$, 
$$G_{\beta}^{\nu}f-G_{\beta}f=K_{\beta}((1-I_{\nu}K_{\beta})^{-1}I_{\nu}G_{\beta}f), \quad f\in L^2({\mathbb R}^d).$$
\end{lem}

\begin{proof}
As in Lemma \ref{lem:bdd}, we may assume that $\nu^{-}=0$. Fix $\beta>\beta_0$ and $f\in L^2({\mathbb R}^d)$. 
Then by Lemma \ref{lem:bdd}, 
we can define the bounded linear operator $(1-I_{\nu}K_{\beta})^{-1}$ on $L^2(\nu)$ 
and $u=(1-I_{\nu}K_{\beta})^{-1}I_{\nu}G_{\beta} f\in L^2(\nu)$. 
For any $v\in {\cal F}$, 
\begin{equation*}
\begin{split}
{\cal E}_{\beta}^{\nu}(G_{\beta}f+K_{\beta}u,v)
&={\cal E}_{\beta}(G_{\beta}f+K_{\beta}u,v)
-\int_{{\mathbb R}^d}I_{\nu}(G_{\beta}f+K_{\beta}u)\cdot I_{\nu}v\,{\rm d}\nu\\
&=\int_{{\mathbb R}^d}fv\,{\rm d}x+\int_{{\mathbb R}^d} u \cdot I_{\nu}v\,{\rm d}\nu
-\int_{{\mathbb R}^d}I_{\nu}(G_{\beta}f+K_{\beta}u)\cdot I_{\nu}v\,{\rm d}\nu.
\end{split}
\end{equation*}
Since $I_{\nu}(G_{\beta}f+K_{\beta}u)=u$, 
we have ${\cal E}_{\beta}^{\nu}(G_{\beta}f+K_{\beta}u,v)=\int_{{\mathbb R}^d}fv\,{\rm d}x$ 
so that the proof is complete. \qed
\end{proof}

\begin{lem}\label{lem:cpt}
Let $\nu^+$ and $\nu^{-}$ be finite Kato class measures, 
and let $\beta_0$ be as in Lemma {\rm \ref{lem:bdd}}. 
\begin{enumerate}
\item[{\rm (i)}] For any $\beta>0$, $K_{\beta}$ is a compact operator 
from $L^2(|\nu|)$ to $L^2({\mathbb R}^d)$.  
\item[{\rm (ii)}] For any $\beta>\beta_0$, 
$G_{\beta}^{\nu}-G_{\beta}$ is a compact operator on $L^2({\mathbb R}^d)$.
\end{enumerate}\end{lem}

\begin{proof}
As in the previous lemmas, we may assume that $\nu^{-}=0$. 
We first prove (i). For $n=1,2,3,\dots$, we define 
$$K_{\beta}^{(n)} f(x)=\int_{|x-y|\geq 1/n}G_{\beta}(x,y)f(y)\,\nu({\rm d}y), \quad f\in L^2(\nu).$$
Since \eqref{eq:res-inf} yields 
$$\iint_{{\mathbb R}^d\times{\mathbb R}^d}
\left(G_{\beta}(x,y){\bf 1}_{|x-y|\geq 1/n}\right)^2\,\nu({\rm d}y){\rm d}x
\leq c_1\nu({\mathbb R}^d)\int_{|x|\geq 1/n}\frac{{\rm d}x}{|x|^{2(d+\alpha)}}<\infty,$$
$K_{\beta}^{(n)}$ is a Hilbert-Schmidt operator so that 
it is compact from $L^2(\nu)$ to $L^2({\mathbb R}^d)$ 
(see, e.g., \cite[Corollary 4.6]{HS78}). 

We now assume that $d>\alpha$. 
Let $\varepsilon\in (0,\alpha/2)$ and 
$$k_1^{(n)}(x,y)=\frac{{\bf 1}_{\{|x-y|<1/n\}}}{|x-y|^{d/2-\varepsilon}}, 
\quad k_2^{(n)}(x,y)=\frac{{\bf 1}_{\{|x-y|<1/n\}}}{|x-y|^{d/2-\alpha+\varepsilon}}.$$
By \eqref{eq:res-0}, there exists $c_1>0$ such that 
for any $n=1,2,3,\dots$, 
$$G_{\beta}(x,y){\bf 1}_{\{|x-y|<1/n\}}\leq c_1 k_1^{(n)}(x,y)k_2^{(n)}(x,y), \ x,y\in {\mathbb R}^d$$
and 
$$K_{\beta}f(x)-K_{\beta}^{(n)}f(x)
=\int_{|x-y|<1/n}G_{\beta}(x,y)f(y)\,\nu({\rm d}y), \ x\in {\mathbb R}^d.$$
Therefore, 
\begin{equation}\label{eq:norm}
\begin{split}
&\|K_{\beta}f-K_{\beta}^{(n)}f\|_{L^2({\mathbb R}^d)}^2
=\int_{{\mathbb R}^d}\left(\int_{|x-y|<1/n}G_{\beta}(x,y)f(y)\,\nu({\rm d}y)\right)^2\,{\rm d}x\\
&\leq c_2\int_{{\mathbb R}^d}
\left(\int_{{\mathbb R}^d}k_1^{(n)}(x,y)^2f(y)^2\,\nu({\rm d}y)\right)
\left(\int_{{\mathbb R}^d}k_2^{(n)}(x,y)^2\,\nu({\rm d}y)\right)\,{\rm d}x\\
&\leq c_2\left\{\int_{{\mathbb R}^d}
\left(\int_{{\mathbb R}^d}k_1^{(n)}(x,y)^2f(y)^2\,\nu({\rm d}y)\right)
\,{\rm d}x\right\}
\sup_{x\in {\mathbb R}^d}\left(\int_{{\mathbb R}^d}k_2^{(n)}(x,y)^2\,\nu({\rm d}y)\right).
\end{split}
\end{equation}

Let $\omega_d=2\pi^{d/2}\Gamma(d/2)^{-1}$ be the surface area of the unit ball in ${\mathbb R}^d$. 
Then by the Fubini theorem,
\begin{equation}\label{eq:k1}
\begin{split}
&\int_{{\mathbb R}^d}
\left(\int_{{\mathbb R}^d}k_1^{(n)}(x,y)^2f(y)^2\,\nu({\rm d}y)\right)\,{\rm d}x\\
&=\int_{{\mathbb R}^d}
\left(\int_{{\mathbb R}^d}k_1^{(n)}(x,y)^2\,{\rm d}x\right)
\,f(y)^2\nu({\rm d}y)
=\frac{\omega_d\|f\|_{L^2(\nu)}^2}{2\varepsilon n^{2\varepsilon}}.
\end{split}
\end{equation}
Since $\alpha>2\varepsilon$, we have by \eqref{eq:res-0},
\begin{equation*}
\begin{split}
&\int_{{\mathbb R}^d}k_2^{(n)}(x,y)^2\,\nu({\rm d}y)
=\int_{|x-y|<1/n}|x-y|^{-d+2\alpha-2\varepsilon}\,\nu({\rm d}y)\\
&\leq \frac{1}{n^{\alpha-2\varepsilon}}\int_{|x-y|<1/n}|x-y|^{-d+\alpha}\,\nu({\rm d}y)
\leq \frac{c_3}{n^{\alpha-2\varepsilon}}\int_{{\mathbb R}^d}G_1(x,y)\,\nu({\rm d}y).
\end{split}
\end{equation*}
Hence  
$$\sup_{x\in {\mathbb R}^d}\int_{{\mathbb R}^d}k_2^{(n)}(x,y)^2\,\nu({\rm d}y)
\leq \frac{c_3}{n^{\alpha-2\varepsilon}}\|G_1{\nu}\|_{\infty}.$$
Combining this inequality with \eqref{eq:k1}, 
we see by \eqref{eq:norm} that 
$$\|K_{\beta}-K_{\beta}^{(n)}\|:=\sup_{f\in L^2(\nu),  f\ne 0}
\frac{\|K_{\beta}f-K_{\beta}^{(n)}f\|_{L^2({\mathbb R}^d)}}{\|f\|_{L^2(\nu)}}
\leq \frac{c_4}{\varepsilon n^{\alpha/2}}\rightarrow 0 \quad (n\rightarrow\infty).$$
For $d\leq \alpha$, we also have 
$\|K_{\beta}-K_{\beta}^{(n)}\|\rightarrow 0$ as $n\rightarrow\infty$ 
by \eqref{eq:res-0} and direct calculation. 
Since $K_{\beta}^{(n)}$ is compact, 
so is $K_{\beta}$ by \cite[Theorem VI.12]{RS4}. 
This completes the proof of  (i).

Since $G_{\beta}$ is a bounded linear operator from 
$L^2({\mathbb R}^d)$ to $({\cal F},\sqrt{{\cal E}_{\beta}})$ 
and $(1-I_{\nu}K)^{-1}I_{\nu}$ is a bounded linear operator from $({\cal F},\sqrt{{\cal E}_{\beta}})$ to $L^2(\nu)$, 
Lemma \ref{lem:resol-1} and (i) imply (ii). 
\qed
\end{proof}

\noindent
{\it Proof of Proposition {\rm \ref{prop:ess}}.} 
Since $\sigma_{\rm ess}((-\Delta)^{\alpha/2}/2)=[0,\infty)$, 
the assertion  follows by Lemma \ref{lem:cpt} and \cite[Theorem VIII.14]{RS4}. 
\qed

\begin{rem}\rm 
Let $\nu^{+}$ and $\nu^{-}$ be Kato class measures such that 
$\nu=\nu^{+}-\nu^{-}$ forms a signed Borel measure on ${\mathbb R}^d$. 
If $\tilde{\nu}^+-\tilde{\nu}^{-}$ is the Jordan decomposition of $\nu$, 
then $\tilde{\nu}^+$ and $\tilde{\nu}^{-}$ are also Kato class measures and 
$A_t^{\nu^{+}}-A_t^{{\nu}^{-}}=A_t^{\tilde{\nu}^+}-A_t^{\tilde{\nu}^{-}}$ 
by the uniqueness of the Revuz correspondence (\cite[Theorem 5.13]{FOT11}). 
In particular, Proposition \ref{prop:ess} is true as it is 
even if $\nu=\nu^{+}-\nu^{-}$ is not the Jordan decomposition of $\nu$.
\end{rem}

We next discuss the asymptotic behavior of an integral 
associated with the ground state of ${\cal H}^{\nu}$. 
In what follows, 
we may and do assume that $\nu$ can be decomposed as $\nu=\nu^+-\nu^{-}$ 
for some $\nu^{+}, \nu^{-}\in {\cal K}_{\infty}(1)$. 
Let $\lambda(\nu)$ be the bottom of the $L^2$-spectrum of ${\cal H}^{\nu}$. 
Then 
$$
\lambda(\nu)
=\inf\left\{{\cal E}(u,u)-\int_{{\mathbb R}^d}u^2\,{\rm d}\nu
\mid u\in C_0^{\infty}({\mathbb R}^d), \int_{{\mathbb R}^d}u^2\,{\rm d}x=1\right\},
$$
where $C_0^{\infty}({\mathbb R}^d)$ is the totality of smooth functions 
with compact support in ${\mathbb R}^d$.
Moreover, if $\lambda(\nu)<0$, then $\lambda(\nu)$ is the eigenvalue 
and the corresponding eigenfunction, which is called the ground state, 
has a bounded and strictly positive continuous version 
(\cite[Theorem 2.8 and Section 4]{T08}).  
We write $h$ for this version with $L^2$-normalization $\|h\|_{L^2({\mathbb R}^d)}=1$. 

By the same proof as for \cite[Lemma 4.1]{T08} and \cite[Lemma A.1]{S19}, 
we see that for any positive constants $p$ and $p'$ with $p'<1<p$, 
there exist positive constants $c$ and $C$ such that 
\begin{equation}\label{eq:ground-est}
\frac{c}{|x|^{(d+\alpha)p}}\leq h(x)\leq \frac{C}{|x|^{(d+\alpha)p'}} \quad (|x|\geq 1).  
\end{equation}If $\nu^{+}$ and $\nu^{-}$ are in addition compactly supported in ${\mathbb R}^d$, 
then \eqref{eq:ground-est} is valid for $p=p'=1$. 

Let $\lambda:=\lambda(\nu)$ and 
\begin{equation}\label{eq:const}
c_{\star}=\frac{C_{d,\alpha}\omega_d}{\alpha (-\lambda)^2}
=\frac{\sin(\pi\alpha/2)\Gamma((d+\alpha)/2)\Gamma(\alpha/2)}{(-\lambda)^2 2^{1-\alpha}\pi\Gamma(d/2)}.
\end{equation}
The next lemma determines the asymptotic behavior 
of the ground state $h$ integrated outside the ball.
\begin{lem}\label{lem:ground}
Suppose that $\lambda<0$. 
\begin{enumerate}
\item[{\rm (i)}] 
For any $x\in {\mathbb R}^d$,
$$h(x)=\int_{{\mathbb R}^d}G_{-\lambda}(x,y)h(y)\,\nu({\rm d}y).$$
\item[{\rm (ii)}] 
If $\nu^{+}$ and $\nu^{-}$ are compactly supported in ${\mathbb R}^d$, 
then $\int_{{\mathbb R}^d}h(y)\,\nu({\rm d}y)>0$ and 
\begin{equation}\label{eq:lim-tail}
R^{\alpha}\int_{|y|>R}h(y)\,{\rm d}y \rightarrow 
c_{\star}\int_{{\mathbb R}^d}h(y)\,\nu({\rm d}y) 
\quad (R\rightarrow\infty).
\end{equation}
\end{enumerate}
\end{lem}

\begin{proof}
Let $\nu^+$ and $\nu^{-}$ belong to ${\cal K}_{\infty}(1)$ and $\lambda<0$. 
Since \cite[Lemma 3.1 (i)]{NS21+} and its proof remain valid under the current setting,  
we have (i) in the same way as for the proof of \cite[Lemma 3.1 (iii)]{NS21+}.

We assume in addition that $\nu^+$ and $\nu^{-}$ are compactly supported in ${\mathbb R}^d$. 
Then by \eqref{eq:ground-est} with $p=p'=1$, $\int_{|y|>R}h(y)\,{\rm d}y$ is convergent for any $R>0$. 
Since (i) yields 
\begin{equation*}
\begin{split}
\int_{|y|>R}h(y)\,{\rm d}y
&=\int_{|y|>R}\left(\int_{{\mathbb R}^d}G_{-\lambda}(y,z)h(z)\,\nu^+({\rm d}z)\right)\,{\rm d}y\\
&-\int_{|y|>R}\left(\int_{{\mathbb R}^d}G_{-\lambda}(y,z)h(z)\,\nu^{-}({\rm d}z)\right)\,{\rm d}y,
\end{split}
\end{equation*}
we have by \eqref{eq:res-inf},
\begin{equation*}
\begin{split}
\int_{|y|>R}\left(\int_{{\mathbb R}^d}G_{-\lambda}(y,z)h(z)\,\nu^{\pm}({\rm d}z)\right)\,{\rm d}y
&\sim \frac{C_{d,\alpha}}{(-\lambda)^2}\int_{|y|>R}\frac{{\rm d}y}{|y|^{d+\alpha}}
\int_{{\mathbb R}^d}h(z)\,\nu^{\pm}({\rm d}z)\\
&=\frac{c_{\star}}{R^{\alpha}}
\int_{{\mathbb R}^d}h(z)\,\nu^{\pm}({\rm d}z),
\end{split}
\end{equation*}
whence \eqref{eq:lim-tail} holds. 
Moreover, since  
there exist $c_1>0$ and $c_2>0$ by \eqref{eq:ground-est}  
such that $c_1\leq R^{\alpha}\int_{|y|>R}h(y)\,{\rm d}y\leq c_2$ for any $R\geq 1$, 
we obtain $\int_{{\mathbb R}^d}h(y)\,\nu({\rm d}y)>0$.  \qed
\end{proof}

For $\alpha=2$, we proved in \cite[Lemma 3.1 (iv)]{NS21+} 
the assertion corresponding to Lemma \ref{lem:ground} (ii), 
but the scaling order there is exponential in contrast with the polynomial order in \eqref{eq:lim-tail}.

Suppose that $\nu^{+}$ and $\nu^{-}$ are Kato class measures with compact support in ${\mathbb R}^d$ 
and $\lambda<0$. 
Then for any $\beta>-\lambda$, we have
$$
\inf\left\{{\cal E}_{\beta}(u,u)-\int_{{\mathbb R}^d}u^2\,{\rm d}\nu
\mid u\in C_0^{\infty}({\mathbb R}^d), \int_{{\mathbb R}^d}u^2\,{\rm d}x=1\right\}
=\beta+\lambda>0
$$
so that by \cite[Lemma 3.5]{T02},
$$
\inf\left\{{\cal E}_{\beta}(u,u)+\int_{{\mathbb R}^d}u^2\,{\rm d}\nu^{-}
\mid u\in C_0^{\infty}({\mathbb R}^d), \int_{{\mathbb R}^d}u^2\,{\rm d}\nu^{+}=1\right\}>1.
$$
Hence by \cite[Lemma 3.5 (1), Theorems 3.6 and 5.2]{C02} 
with Lemma \ref{lem:resol} (ii), 
there exist positive constants 
$c$ and $C$ for any $\beta>-\lambda$ such that 
\begin{equation}\label{eq:res-comp}
cG_{\beta}(x,y)\leq G_{\beta}^{\nu}(x,y)\leq C G_{\beta}(x,y).
\end{equation}

Let $\lambda_2(\nu)$ be the second bottom of the spectrum for ${\cal H}^{\nu}$,
$$\lambda_2(\nu)=\inf\left\{{\cal E}(u,u)-\int_{{\mathbb R}^d}u^2\,{\rm d}\nu 
\mid u\in C_0^{\infty}({\mathbb R^d}), 
\int_{{\mathbb R}^d}u^2\,{\rm d}x=1, \int_{{\mathbb R}^d}uh\,{\rm d}x=0\right\}.$$
If $\lambda<0$, then Proposition \ref{prop:ess} implies that 
$(\lambda<)\lambda_2(\nu)\leq 0$.
Then as in \cite[Lemma 3.1 (ii)]{NS21+}, there exists $C>0$ such that 
\begin{equation}\label{eq:gap}
|p_t^{\nu}(x,y)-e^{-\lambda t}h(x)h(y)|\leq Ce^{-\lambda_2 t}, \quad t\geq 1.
\end{equation}

\subsection{Asymptotic behaviors of Feynman-Kac functionals}\label{subsect:asymp-fk}
In this subsection, we prove the asymptotic properties of 
the Feynman-Kac functionals for a symmetric stable process.  
Even though our approach is similar to that of \cite[Proposition 3.2]{NS21+}, 
we need calculations by taking into account
the polynomial decay property of the tail distribution 
of the symmetric stable process  in \eqref{eq:bg-asymp}. 
Throughout this subsection, 
we assume that $\nu^+$ and $\nu^-$ are 
the Kato class measures with compact support in ${\mathbb R}^d$ 
and that $\lambda<0$.  

Let 
\begin{equation}\label{eq:decomp-0}
q_t(x,y)=p_t^{\nu}(x,y)-p_t(x,y)-e^{-\lambda t}h(x)h(y)
\end{equation}
so that for $R>0$, 
\begin{equation}\label{eq:decomp}
E_x[e^{A_t^{\nu}};|X_t|>R]
=P_x(|X_t|>R)+e^{-\lambda t}h(x)\int_{|y|>R}h(y)\,{\rm d}y+\int_{|y|>R}q_t(x,y)\,{\rm d}y.
\end{equation}
For $r>0$, let $B(r)=\{y\in {\mathbb R}^d \mid |y|<r\}$ be an open ball 
with radius $r$ centered at the origin. 
Fix $M>0$ so that the support of $|\nu|$ is included in $B(M)$. 
For $c>0$, define 
$$I_c(t,R)=
\begin{cases}
e^{ct}/R^{\alpha} & (\lambda_2<0), \\
(tP_0(|X_t|>R-M))\wedge (e^{ct}/R^{\alpha}) & (\lambda_2=0)
\end{cases}
$$
and
$$J(t,R)=e^{-\lambda t}(R-M)^d
\int_{t^{1/\alpha}}^{\infty}e^{\lambda u^{\alpha}}g\left(\frac{R-M}{u}\right)\frac{{\rm d}u}{u^{d+1}},$$
where $g$ is the same function as in \eqref{eq:scaling}.

\begin{prop}\label{prop:q}
For any $c>0$ with $c>-\lambda_2$, 
there exists $C>0$ such that for any $x\in {\mathbb R}^d$, $t\geq 1$ and $R>2M$, 
$$\left|\int_{|y|>R}q_t(x,y)\,{\rm d}y\right|
\leq C\left(h(x)P_0(|X_t|>R-M)+I_c(t,R)+h(x)J(t,R)\right).$$
\end{prop}
\begin{proof}
As for \cite[(3.19)]{NS21+}, 
we have 
\begin{equation}\label{eq:q-div}
\begin{split}
\int_{|y|>R}q_t(x,y)\,{\rm d}y
&=\int_0^1\left(\int_{{\mathbb R}^d}p_s^{\nu}(x,z)P_z(|X_{t-s}|>R)\,\nu({\rm d}z)\right)\,{\rm d}s\\
&+\int_1^t
\left(\int_{{\mathbb R}^d}(p_s^{\nu}(x,z)-e^{-\lambda s}h(x)h(z))P_z(|X_{t-s}|>R)\,\nu({\rm d}z)\right)\,{\rm d}s\\
&-e^{-\lambda t}h(x)\int_{t-1}^{\infty}e^{\lambda s}
\left(\int_{{\mathbb R}^d}h(z)P_z(|X_s|>R)\,\nu({\rm d}z)\right){\rm d}s\\
&={\rm (I)}+{\rm (II)}-{\rm (III)}.
\end{split}
\end{equation}
For any $s\in [0,t]$ and $z\in {\mathbb R}^d$, 
\begin{equation}\label{eq:comp}
P_z(|X_{t-s}|>R)\leq P_0(|X_{t-s}|>R-|z|)\leq P_0(|X_t|>R-|z|)
\end{equation}
by the spatial uniformity and scaling property of the symmetric stable process.
Then for any $\varepsilon>0$,  we see  
by \eqref{eq:res-inf}, 
\eqref{eq:ground-est} (with $p=p'=1$) and \eqref{eq:res-comp} (with $\beta=-\lambda+\varepsilon$) 
that 
\begin{equation*}
\begin{split}
&\int_0^1\left(\int_{{\mathbb R}^d}p_s^{\nu}(x,z)\,|\nu|({\rm d}z)\right)\,{\rm d}s
\leq e^{-\lambda +\varepsilon}
\int_{{\mathbb R}^d}
\left(\int_0^1e^{(\lambda-\varepsilon)s}p_s^{\nu}(x,z)\,{\rm d}s\right)\,|\nu|({\rm d}z)\\
&\leq e^{-\lambda +\varepsilon}
\int_{{\mathbb R}^d}G_{-\lambda+\varepsilon}^{\nu}(x,z)\,|\nu|({\rm d}z)
\leq c_1\int_{{\mathbb R}^d}G_{-\lambda+\varepsilon}(x,z)\,|\nu|({\rm d}z)
\leq c_2h(x).
\end{split}
\end{equation*}
Hence  by \eqref{eq:comp},
\begin{equation*}
\begin{split}
{\rm (I)}
&\leq P_0(|X_t|>R-M)\int_0^1\left(\int_{{\mathbb R}^d}p_s^{\nu}(x,z)\,|\nu|({\rm d}z)\right)\,{\rm d}s\\
&\leq c_3 P_0(|X_t|>R-M)h(x).
\end{split}
\end{equation*}

Fix $c>0$ with $c\geq -\lambda_2$. 
Then by \eqref{eq:comp} and Lemma \ref{lem:resol} (i),
\begin{equation*}
\begin{split}
&\int_1^t\left(\int_{{\mathbb R}^d}e^{-\lambda_2 s}P_z(|X_{t-s}|>R)\,|\nu|({\rm d}z)\right)\,{\rm d}s\\
&\leq |\nu|({\mathbb R}^d)\int_{1}^te^{-\lambda_2 s}P_0(|X_{t-s}|>R-M)\,{\rm d}s\\
&\leq  |\nu|({\mathbb R}^d)e^{ct}\int_{|y|>R-M}G_c(0,y)\,{\rm d}y
\leq c_4 e^{ct}\int_{|y|>R-M}\frac{{\rm d}y}{|y|^{d+\alpha}}
\leq \frac{c_5 e^{ct}}{R^{\alpha}}.
\end{split}
\end{equation*}
If $\lambda_2=0$, then by \eqref{eq:comp} again,
$$\int_1^t\left(\int_{{\mathbb R}^d}P_z(|X_{t-s}|>R)\,|\nu|({\rm d}z)\right)\,{\rm d}s
\leq c_6 tP_0(|X_t|>R-M).$$
Hence by \eqref{eq:gap}, 
\begin{equation*}
\begin{split}
|{\rm (II)}|
&\leq \int_1^t
\left(\int_{{\mathbb R}^d}|p_s^{\nu}(x,z)-e^{-\lambda s}h(x)h(z)|P_z(|X_{t-s}|>R)\,|\nu|({\rm d}z)\right)\,{\rm d}s\\
&\leq c_7\int_1^t
\left(\int_{{\mathbb R}^d}e^{-\lambda_2 s}P_z(|X_{t-s}|>R)\,|\nu|({\rm d}z)\right)\,{\rm d}s
\leq c_8 I_c(t,R).
\end{split}
\end{equation*}

Since \eqref{eq:scaling} yields
\begin{equation}\label{eq:prob-r}
P_0(|X_s|>R)
=\omega_d\int_{R/{s^{1/\alpha}}}^{\infty}g(r)r^{d-1}\,{\rm d}r
=\omega_d R^d\int_0^{s^{1/\alpha}}g\left(\frac{R}{u}\right)\frac{{\rm d}u}{u^{d+1}}, 
\end{equation}
we have by \eqref{eq:prob-r}  and integration by parts formula,
$$\int_t^{\infty}e^{\lambda s}P_0(|X_s|>R)\,{\rm d}s
=\frac{e^{\lambda t}}{-\lambda}P_0(|X_t|>R)
+\frac{\omega_d}{-\lambda}R^d
\int_{t^{1/\alpha}}^{\infty}e^{\lambda u^{\alpha}}g\left(\frac{R}{u}\right)\frac{{\rm d}u}{u^{d+1}}.$$
We also see by \eqref{eq:comp} that 
\begin{equation*}
\begin{split}
\int_{t-1}^{\infty}e^{\lambda s}P_0(|X_s|>R-M)\,{\rm d}s
&=\int_t^{\infty}e^{\lambda (s-1)}P_0(|X_{s-1}|>R-M)
\,{\rm d}s\\
&\leq e^{-\lambda}\int_t^{\infty}e^{\lambda s}P_0(|X_s|>R-M)\,{\rm d}s. 
\end{split}
\end{equation*}
Therefore, 
\begin{equation*}
\begin{split}
|{\rm (III)}|
&\leq e^{-\lambda t}h(x)\int_{t-1}^{\infty}e^{\lambda s}P_0(|X_s|>R-M)
\,{\rm d}s\left(\int_{{\mathbb R}^d}h(z)\,|\nu|({\rm d}z)\right)\\
&\leq c_9 e^{-\lambda t}h(x)\int_t^{\infty}e^{\lambda s}P_0(|X_s|>R-M)
\,{\rm d}s\\
&\leq c_{10} h(x)\left(P_0(|X_t|>R-M)+J(t,R)\right),
\end{split}
\end{equation*}
which completes the proof.
\qed
\end{proof}

\begin{rem}\rm 
Let us take $R=0$ in \eqref{eq:q-div}. 
Then by \eqref{eq:gap} and \eqref{eq:decomp-0}, 
there exists $c>0$ such that for any $f\in {\cal B}_b({\mathbb R}^d)$ and $t\geq 1$, 
\begin{equation*}
\sup_{x\in {\mathbb R}^d}
\left|e^{\lambda t}E_x\left[e^{A_t^{\nu}}f(X_t)\right]-h(x)\int_{{\mathbb R}^d}f(y)h(y)\,{\rm d}y\right|
\leq c\|f\|_{\infty}e^{\lambda t}\left(t\vee e^{-\lambda_2(\nu)t}\right).
\end{equation*}
The right hand side above goes to $0$ as $t\rightarrow\infty$ 
because $\lambda<\lambda_2(\nu)\leq 0$. 
This result extends the assertion in \cite[Remark 3.4]{NS21+} for the Brownian motion 
to the symmetric stable process, and provides a convergence rate bound in \cite[(1.3)]{T08}.
\end{rem}

Let $R(t)$ be a positive measurable function on $(0,\infty)$ 
such that $R(t)\rightarrow\infty$ as $t\rightarrow\infty$. 
Then by Lemma \ref{lem:ground} (ii), 
\begin{equation}\label{eq:asymp-r}
\eta(t):=e^{-\lambda t}\int_{|y|>R(t)}h(y)\,{\rm d}y
\sim c_*\frac{e^{-\lambda t}}{R(t)^{\alpha}}
\end{equation}
with 
\begin{equation}\label{eq:c-ast}
c_*=c_{\star}\int_{{\mathbb R}^d}h(y)\,\nu({\rm d}y).
\end{equation}
The next lemma reveals the exact asymptotic behavior 
of the Feynman-Kac semigroup conditioned that 
the particle at time $t$ sits outside the ball with radius $R(t)$.

\begin{lem}\label{lem:fk-asymp}
Let $K$ be a compact set in ${\mathbb R}^d$. 
If $R(t)/t^{1/\alpha}\rightarrow \infty$ as $t\rightarrow\infty$, 
then there exist positive constants $c_1$, $c_2$ and $T$ such that 
for any $x\in K$, $t\geq T$ and $s\in [0,t-1]$,
$$E_x\left[e^{A_{t-s}^{\nu}};|X_{t-s}|>R(t)\right]
=e^{\lambda s}h(x)\eta(t)(1+\theta_{s,x}(t))
$$
with 
$$|\theta_{s,x}(t)|\leq c_1e^{-c_2(t-s)}.$$
Here $c_1$ and $c_2$ can be independent of the choice of the function $R(t)$. 
In particular, 
$$\lim_{t\rightarrow\infty}
\sup_{x\in K}\left|\frac{1}{h(x)\eta(t)}E_x\left[e^{A_t^{\nu}};|X_t|>R(t)\right]-1\right|
= 0.$$
\end{lem}

\begin{proof}
Take $M>0$ so that $B(M)$ includes both $K$ and the support of $|\nu|$. 
Then for any $s\in [0,t-1]$,
$$(R(t)-M)/(t-s)^{1/\alpha}\geq (R(t)-M)/t^{1/\alpha}$$
and the right hand above goes to $\infty$ as $t\rightarrow\infty$. 
Hence by \eqref{eq:scaling} and \eqref{eq:bg-asymp}, 
there exist $c_1>0$, $c_2>0$ and  $T_1>1$ such that 
for any $x\in K$ and $t\geq T_1$ and $s\in [0,t-1]$, 
\begin{equation}\label{eq:bound-prob}
\begin{split}
&P_x(|X_{t-s}|>R(t))
\leq P_0(|X_{t-s}|>R(t)-M)
=\omega_d\int_{\frac{R(t)-M}{(t-s)^{1/\alpha}}}^{\infty}g(u)u^{d-1}\,{\rm d}u\\
&\leq c_1\int_{\frac{R(t)-M}{(t-s)^{1/\alpha}}}^{\infty}\frac{{\rm d}u}{u^{\alpha+1}}
\leq c_2\frac{t-s}{R(t)^{\alpha}}
=c_2e^{\lambda(t-s)}(t-s)\frac{e^{-\lambda(t-s)}}{R(t)^{\alpha}}.
\end{split}
\end{equation}

For any $c>0$,  
\begin{equation}\label{eq:i-bound}
I_c(t-s,R(t))\leq \frac{e^{c(t-s)}}{R(t)^{\alpha}}
=e^{(c+\lambda)(t-s)}\frac{e^{-\lambda(t-s)}}{R(t)^{\alpha}}.
\end{equation}
By \eqref{eq:bg-asymp}, there exists $T_2>1$ such that for all $t\geq T_2$,
$$
\int_{t^{1/\alpha}}^{R(t)}e^{\lambda u^{\alpha}}g\left(\frac{R(t)-M}{u}\right)\frac{{\rm d}u}{u^{d+1}}
\leq \frac{c_3}{R(t)^{d+\alpha}}
\int_{t^{1/\alpha}}^{\infty}e^{\lambda u^{\alpha}}u^{\alpha-1}\,{\rm d}u
\leq \frac{c_4e^{\lambda t}}{R(t)^{d+\alpha}}
$$
and
$$
\int_{R(t)}^{\infty}e^{\lambda u^{\alpha}}g\left(\frac{R(t)-M}{u}\right)\frac{{\rm d}u}{u^{d+1}}
\leq c_5\int_{R(t)}^{\infty}e^{\lambda u^{\alpha}}\frac{{\rm d}u}{u^{d+1}}
\leq \frac{c_6e^{\lambda R(t)^{\alpha}}}{R(t)^{d+\alpha}}
\leq \frac{c_7e^{\lambda t}}{R(t)^{d+\alpha}}.
$$
Hence
\begin{equation}\label{eq:j-bound}
\begin{split}
J(t,R(t))
&=e^{-\lambda t}(R(t)-M)^d
\int_{t^{1/\alpha}}^{\infty}e^{\lambda u^{\alpha}}g\left(\frac{R(t)-M}{u}\right)\frac{{\rm d}u}{u^{d+1}}\\
&\leq \frac{c_8}{R(t)^{\alpha}}
=c_8e^{\lambda(t-s)}\frac{e^{-\lambda(t-s)}}{R(t)^\alpha}.
\end{split}
\end{equation}
Note that all the constants $c_i$ can 
be independent of the choice of the function $R(t)$. 

Fix $c\in (-\lambda_2(\nu),-\lambda)$. 
Then by combining \eqref{eq:decomp} and Proposition \ref{prop:q} with 
\eqref{eq:bound-prob}--\eqref{eq:j-bound},  
there exist positive constants $c_{9}$, $c_{10}$ and $c_{11}$, and $T\geq 1$ 
such that for any $x\in K$, $t\geq T$ and $s\in [0,t-1]$,
\begin{equation*}
\begin{split}
&\left|E_x\left[e^{A_{t-s}^{\nu}};|X_{t-s}|>R(t)\right]-e^{\lambda s}\eta(t)h(x)\right|\\
&\leq P_x(|X_{t-s}|>R(t))+\left|\int_{|y|>R(t)}q_{t-s}(x,y)\,{\rm d}y\right|\\
&\leq \frac{c_{9}e^{-\lambda (t-s)}}{R(t)^{\alpha}}\left(e^{\lambda(t-s)}(t-s)
+e^{(c+\lambda)(t-s)}+e^{\lambda(t-s)}\right)
\leq c_{10}e^{\lambda s}e^{-c_{11}(t-s)}\frac{e^{-\lambda t}}{R(t)^{\alpha}}.
\end{split}
\end{equation*}
Then by \eqref{eq:asymp-r}, the proof is complete.
\qed
\end{proof}

Recall that by \cite[Lemma 3.4]{S08}, 
we have for any $\mu\in {\cal K}_{\infty}(1)$, 
\begin{equation}\label{eq:gauge}
\sup_{x\in {\mathbb R}^d}E_x\left[\int_0^{\infty}e^{2\lambda s+A_s^{\nu}}\,{\rm d}A_s^{\mu}\right]<\infty.
\end{equation}
 The next two lemmas will be used later for the second moment estimates 
of the expected population for a branching symmetric stable process.
\begin{lem}\label{lem:second}
Let $K$ be a compact set in ${\mathbb R}^d$ 
and $\mu$ a Kato class measure with compact support in ${\mathbb R}^d$. 
If $R(t)/t^{1/\alpha}\rightarrow \infty$ as $t\rightarrow\infty$, 
then there exist $C>0$ and $T>0$ such that 
for any $t\geq T$,
$$\sup_{x\in K}E_x\left[\int_0^t e^{A_s^{\nu}}
E_{X_s}\left[e^{A_{t-s}^{\nu}};|X_{t-s}|>R(t)\right]^2\,{\rm d}A_s^{\mu}\right]
\leq C\eta(t)^2.
$$
\end{lem}

\begin{proof} 
Fix $x\in K$. For $t\geq 1$,
\begin{equation}\label{eq:second-1}
\begin{split}
&E_x\left[\int_0^t e^{A_s^{\nu}}E_{X_s}\left[e^{A_{t-s}^{\nu}};|X_{t-s}|>R(t)\right]^2
\,{\rm d}A_s^{\mu}\right]\\
&=E_x\left[\int_0^{t-1}e^{A_s^{\nu}}E_{X_s}\left[e^{A_{t-s}^{\nu}};|X_{t-s}|>R(t)\right]^2
\,{\rm d}A_s^{\mu}\right]\\
&+E_x\left[\int_{t-1}^t e^{A_s^{\nu}}E_{X_s}\left[e^{A_{t-s}^{\nu}};|X_{t-s}|>R(t)\right]^2
\,{\rm d}A_s^{\mu}\right]
={\rm (IV)}+{\rm (V)}.
\end{split}
\end{equation}
If $0\leq s\leq t-1$, then Lemma \ref{lem:fk-asymp} yields for any $z\in \supp[\mu]$
$$
E_z\left[e^{A_{t-s}^{\nu}};|X_{t-s}|>R(t)\right]
\leq c_1e^{\lambda s} \eta(t)
$$
so that  by \eqref{eq:gauge},
\begin{equation}\label{eq:iv}
{\rm (IV)}
\leq c_2 \eta(t)^2 \sup_{x\in {\mathbb R}^d}E_x\left[\int_0^{\infty}e^{2\lambda s+A_s^{\nu}}\,{\rm d}A_s^{\mu}\right]
\leq c_3 \eta(t)^2.
\end{equation}

By \cite[Theorem 6.1 (i)]{ABM91} and \eqref{eq:comp}, 
there exists $c_4>0$ such that for any $M>0$, $R>M$, $t\in [0,1]$ and $x\in {\mathbb R}^d$ with $|x|\leq M$, 
$$E_x\left[e^{A_t^{\nu}};|X_t|>R\right]
\leq c_4 P_0(|X_t|>R-M)\leq c_4 P_0(|X_1|>R-M).$$
Hence \eqref{eq:bg-asymp} implies that 
for any $z\in \supp[\mu]$,  all sufficiently large $t\geq 1$ and  any $s\in [t-1,t]$, 
\begin{equation}\label{eq:large-time}
E_z\left[e^{A_{t-s}^{\nu}};|X_{t-s}|>R(t)\right]
\leq c_5 P_0(|X_1|>R(t)-M)
\leq \frac{c_6}{R(t)^{\alpha}}.
\end{equation}
Since \eqref{eq:gauge} yields 
\begin{equation*}
E_x\left[\int_{t-1}^t e^{A_s^{\nu}}\,{\rm d}A_s^{\mu}\right]
\leq e^{-2\lambda t}
\sup_{x\in {\mathbb R}^d}E_x\left[\int_0^{\infty}e^{2\lambda s+A_s^{\nu}}\,{\rm d}A_s^{\mu}\right]
\leq c_7e^{-2\lambda t},
\end{equation*}
we have by \eqref{eq:asymp-r},
$$
{\rm (V)}\leq 
\frac{c_8}{R(t)^{2\alpha}}E_x\left[\int_{t-1}^t e^{A_s^{\nu}}\,{\rm d}A_s^{\mu}\right]
\leq \frac{c_{9}e^{-2\lambda t}}{R(t)^{2\alpha}}
\leq c_{10}\eta(t)^2.
$$
Combining this with \eqref{eq:second-1} and \eqref{eq:iv}, 
we complete the proof. 
\qed
\end{proof}

For $\kappa>0$, let $R^{\kappa}(t)=(e^{-\lambda t}\kappa)^{1/\alpha}$. 
Letting $R(t)=R^{\kappa}(t)$ in \eqref{eq:asymp-r}, 
we get 
\begin{equation}\label{eq:asymp-r1}
\eta(t)\rightarrow c_*\kappa^{-1} \ (t\rightarrow\infty).
\end{equation}

\begin{lem}\label{lem:second-1}
Let $K\subset {\mathbb R}^d$ be a compact set. 
\begin{enumerate}
\item[{\rm (i)}] For any $\kappa>0$,
$$\lim_{t\rightarrow\infty}\sup_{x\in K}\left|\frac{\kappa}{h(x)}E_x\left[e^{A_t^{\nu}};|X_t|>R^{\kappa}(t)\right]-c_*\right|
=0.$$
\item[{\rm (ii)}] 
Let $\mu$ be a Kato class measure with compact support in ${\mathbb R}^d$. 
Then 
$$\lim_{\kappa\rightarrow\infty}\limsup_{t\rightarrow\infty}\sup_{x\in K}
\kappa E_x\left[\int_0^t e^{A_s^{\nu}}
E_{X_s}\left[e^{A_{t-s}^{\nu}};|X_{t-s}|>R^{\kappa}(t)\right]^2\,{\rm d}A_s^{\mu}\right]=0.$$
\end{enumerate}
\end{lem}

\begin{proof} 
(i) follows by Lemma \ref{lem:fk-asymp} and \eqref{eq:asymp-r1}.
We now show (ii). 
By Lemma \ref{lem:fk-asymp} and \eqref{eq:large-time}, 
there exist $c_1>0$ and $T=T(\kappa)>1$ for any $\kappa>0$   
such that, for any $z\in \supp[\mu]$, $t\geq T$ and $s\in [0,t]$,
$$
E_z\left[e^{A_{t-s}^{\nu}};|X_{t-s}|>R^{\kappa}(t)\right]
\leq c_1e^{\lambda s} \eta(t).
$$
Hence by \eqref{eq:gauge}, 
there exists $c_2>0$ such that for all $t\geq  T$,
$$\sup_{x\in K}\kappa E_x\left[\int_0^t e^{A_s^{\nu}}
E_{X_s}\left[e^{A_{t-s}^{\nu}};|X_{t-s}|>R^{\kappa}(t)\right]^2\,{\rm d}A_s^{\mu}\right]
\leq \frac{c_2}{\kappa}(\kappa\eta(t))^2.
$$
Then by \eqref{eq:asymp-r1},
$$\limsup_{t\rightarrow\infty}\sup_{x\in K}\kappa E_x\left[\int_0^t e^{A_s^{\nu}}
E_{X_s}\left[e^{A_{t-s}^{\nu}};|X_{t-s}|>R^{\kappa}(t)\right]^2\,{\rm d}A_s^{\mu}\right]\leq \frac{c_2c_*^2}{\kappa}.
$$
The right hand side above goes to 0 as $\kappa\rightarrow \infty$. 
\qed
\end{proof}

\section{Maximal displacement of branching symmetric stable processes}
\label{sect:branching}

In this section, we first introduce a model of branching symmetric stable processes. 
We then present our main results with examples. 

\subsection{Branching symmetric stable processes}

For $\alpha\in (0,2)$, let 
${\mathbf M}=(\Omega,{\cal F}, \{X_t\}_{t\geq 0},\{P_x\}_{x\in {\mathbb R}^d}, \{{\cal F}_t\}_{t\geq 0})$ 
be a symmetric $\alpha$-stable process on ${\mathbb R}^d$, 
where $\{{\cal F}_t\}_{t\geq 0}$ is the minimal augmented admissible filtration. 

Let us formulate the model of a branching symmetric $\alpha$-stable process on ${\mathbb R}^d$ 
by following \cite{S08} and references therein. 
We first define the set ${\mathbf X}$ as follows: 
let $({\mathbb R}^d)^{(0)}=\{\Delta\}$ and $({\mathbb R}^d)^{(1)}={\mathbb R}^d$. 
Let $n\geq 2$. 
For ${\bf x}=(x^1,\dots, x^n)$ and ${\bf y}=(y^1, \dots, y^n)$ in $({\mathbb R}^d)^n$, 
we write ${\bf x}\sim {\bf y}$ 
if there exists a permutation $\sigma$ of $\{1,2,\dots, n\}$ such that 
$y^i=x^{\sigma(i)}$ for any $i=1,\dots, n$. 
Using this equivalence relation, 
we define $({\mathbb R}^d)^{(n)}=({\mathbb R}^d)^{n}/\sim$ for $n\geq 2$ 
and ${\mathbf X}=\cup_{n=0}^{\infty}({\mathbb R}^d)^{(n)}$. 

Let ${\mathbf p}=\{p_n(x)\}_{n=0}^{\infty}$ be a probability function on ${\mathbb R}^d$, 
$0\leq p_n(x)\leq 1$ and $\sum_{n=0}^{\infty}p_n(x)=1$ for any $x\in {\mathbb R^d}$. 
We assume that $p_0(x)+p_1(x)\not\equiv 1$ to avoid the triviality. 
Fix $\mu\in {\cal K}$ and ${\mathbf p}$. 
We next introduce a particle system as follows: 
a particle starts from $x\in {\mathbb R}^d$ at time $t=0$ 
and moves by following the distribution $P_x$ 
until the random time $U$. 
Here the distribution of $U$ is given by  
$$P_x(U>t \mid {\cal F}_{\infty})=e^{-A_t^{\mu}} \ (t>0).$$
At time $U$, this particle dies leaving no offspring with probability $p_0(X_{U-})$, 
or splits into $n$ particles with probability $p_n(X_{U-})$ for $n\geq 1$.  
For the latter case, these $n$ particles then move 
by following the distribution $P_{X_{U-}}$ and repeat the same procedure 
independently. 
If there exist $n$ particles alive at time $t$,  
then the positions of these particles 
determine a point in $({\mathbb R}^d)^{(n)}$. 
Let ${\mathbf X}_t$ denote such a point, 
$${\mathbf X}_t=({\mathbf X}_t^{(1)},\dots,{\mathbf X}_t^{(n)})\in ({\mathbb R}^d)^{(n)}.$$
In this way,  we can define the model of
a branching symmetric $\alpha$-stable process 
$\overline{\mathbf M}=(\{{\mathbf X}_t\}_{t\geq 0}, \{{\mathbf P}_{{\mathbf x}}\}_{{\mathbf x}\in {\mathbf X}})$ 
on ${\bf X}$ (or simply on ${\mathbb R}^d$) 
with branching rate $\mu$ and branching mechanism ${\mathbf p}$.
Note that for $x\in {\mathbb R}^d$, ${\mathbf P}_x$ denotes 
the law of the process such that the initial state is a single particle at $x$. 

Let $S$ be the first splitting time of $\overline{\mathbf M}$ given by 
$${\mathbf P}_x(S>t \mid\sigma(X))=P_x(U>t \mid {\cal F}_{\infty})=e^{-A_t^{\mu}} \ (t>0).$$
Let $Z_t$ be the population at time $t$ and $e_0:=\inf\{t>0 \mid Z_t=0\}$ 
the extinction time of $\overline{{\mathbf M}}$. 
Note that $Z_t=0$ for all $t\geq e_0$. 
For $f\in {\cal B}_b({\mathbb R}^d)$, we define 
$$Z_t(f)=
\begin{cases}
\sum_{k=1}^{Z_t} f({\mathbf X}_t^{(k)}) & (t<e_0),\\ 
0 & (t\geq e_0).
\end{cases}$$
For $A\in {\cal B}({\mathbb R}^d)$,  
let $Z_t(A):=Z_t({\bf 1}_A)$ denote the population on $A$ at time $t$.

Let $Q(x)=\sum_{n=0}^{\infty}np_n(x)$ and $\nu_Q({\rm d}x)=Q(x)\mu({\rm d}x)$. 
Let $R(x)=\sum_{n=1}^{\infty}n(n-1)p_n(x)$ and 
$\nu_R({\rm d}x)=R(x)\mu({\rm d}x)$. 
We here recall the next lemma on the first and second moments of $Z_t(f)$:
\begin{lem}\label{lem:moment}{\rm (\cite[Lemma 2.2]{NS21+} and \cite[Lemma 3.3]{S08})}
Let $\mu\in {\cal K}$ and $f\in {\cal B}_b({\mathbb R}^d)$.
\begin{enumerate}
\item[{\rm (i)}] 
If $\nu_Q\in {\cal K}$, then 
$${\mathbf E}_x\left[Z_t(f)\right]=E_x\left[e^{A_t^{(Q-1)\mu}}f(X_t)\right].$$
\item[{\rm (ii)}] 
If $\nu_R\in {\cal K}$, then 
\begin{equation*}
\begin{split}
{\mathbf E}_x\left[Z_t(f)^2\right]
&=E_x\left[e^{A_t^{(Q-1)\mu}}f(X_t)^2\right]\\
&+E_x\left[\int_0^t e^{A_s^{(Q-1)\mu}}E_{X_s}\left[e^{A_{t-s}^{(Q-1)\mu}}f(X_{t-s})\right]^2
\,{\rm d}A_s^{\nu_R}\right].
\end{split}
\end{equation*}
\end{enumerate}
\end{lem}

\subsection{Weak convergence and tail asymptotics}

Let 
$\overline{\mathbf M}=(\{{\mathbf X}_t\}_{t\geq 0}, \{{\mathbf P}_{{\mathbf x}}\}_{{\mathbf x}\in {\mathbf X}})$ 
be a branching symmetric $\alpha$-stable process on ${\mathbb R}^d$  
with branching rate $\mu\in {\cal K}$ and branching mechanism ${\mathbf p}$.
We impose the next assumption on $\mu$ and ${\mathbf p}$:
\begin{assum}\label{assum:weak}
\begin{enumerate}
\item[{\rm (i)}] The support of $\mu$ is compact in ${\mathbb R}^d$.
\item[{\rm (ii)}] $\nu_R\in {\cal K}$. 
\item[{\rm (iii)}] $\lambda((Q-1)\mu)<0$.
\end{enumerate}
\end{assum}

Let $\lambda=\lambda((Q-1)\mu)$. 
Under Assumption \ref{assum:weak}, 
$\lambda$ is the principal eigenvalue of the operator ${\cal H}^{(Q-1)\mu}$ 
on $L^2({\mathbb R}^d)$ as mentioned in Subsection \ref{subsect:spectral}. 
Let $h$ denote the bounded and strictly positive continuous version of 
the corresponding ground state with $L^2$-normalization. 
We define $M_t=e^{\lambda t}Z_t(h)$. 
Then by \cite[Lemma 3.4]{S08}, 
$\{M_t\}_{t\geq 0}$ is a nonnegative 
square integrable ${\mathbf P}_x$-martingale 
so that ${\mathbf E}_x[M_t]=h(x)$ and $M_{\infty}=\lim_{t\rightarrow\infty}M_t$ exists 
${\mathbf P}_x$-a.s. with ${\mathbf P}_x(M_{\infty}>0)>0$.

Let $L_t$ denote the maximal Euclidean norm of 
particles alive at time $t$:
$$L_t=\begin{cases}
\max_{1\leq k\leq Z_t}|{\mathbf X}_t^{(k)}| & (t<e_0),\\
0 & (t\geq e_0).
\end{cases}$$
Since each particle follows the law of the symmetric stable process and 
${\mathbf P}_x(Z_t<\infty)=1$ for any $t>0$, $L_t$ is well-defined 
and ${\mathbf P}_x(L_t<\infty)=1$ for any $t\geq 0$.

In what follows, let $c_*$ denote the positive constant 
given by \eqref{eq:c-ast} with $\nu=(Q-1)\mu$.   
For $\kappa>0$, let $R^{\kappa}(t)=(e^{-\lambda t}\kappa)^{1/\alpha}$. 
We then have 
\begin{thm}\label{thm:weak}
For any $\kappa>0$,
$$\lim_{t\rightarrow\infty}
{\mathbf P}_x(L_t>R^{\kappa}(t))=
{\mathbf E}_x\left[1-\exp\left(-\kappa^{-1}c_*M_{\infty}\right)\right].
$$
\end{thm}

Theorem \ref{thm:weak} extends \cite[Theorem 2.4]{NS21+} 
for the branching Brownian motion to that for the branching symmetric stable process. 
Theorem \ref{thm:weak} implies that $L_t$  grows exponentially fast 
in contrast with the linear growth for the branching Brownian motion.

Since 
$${\mathbf P}_x(L_t>R^{\kappa}(t),e_0<\infty)
\leq {\mathbf P}_x(t<e_0<\infty)\rightarrow 0 \ (t\rightarrow\infty)$$
and $\{e_0<\infty\}\subset \{M_{\infty}=0\}$, 
we obtain 
$$\lim_{t\rightarrow\infty}
{\mathbf P}_x(L_t>R^{\kappa}(t) \mid e_0=\infty)=
{\mathbf E}_x\left[1-\exp\left(-\kappa^{-1}c_*M_{\infty}\right) \mid e_0=\infty\right].
$$
If we let $Y_t=e^{\lambda t/\alpha}L_t$, 
then the equality above reads 
\begin{equation}\label{eq:weak-0}
\lim_{t\rightarrow\infty}
{\mathbf P}_x(Y_t\leq \kappa\mid e_0=\infty)=
{\mathbf E}_x\left[\exp\left(-\kappa^{-\alpha}c_*M_{\infty}\right) \mid e_0=\infty\right].
\end{equation}
Moreover, if $d=1$ and $1<\alpha<2$, 
then \cite[Remark 3.14]{S08} yields 
$\{e_0=\infty\}=\{M_{\infty}>0\},$ ${\mathbf P}_x$-a.s.\ so that 
\begin{equation}\label{eq:weak}
\lim_{t\rightarrow\infty}
{\mathbf P}_x(Y_t\leq \kappa \mid M_{\infty}>0)
={\mathbf E}_x\left[\exp\left(-\kappa^{-\alpha}c_*M_{\infty}\right) \mid M_{\infty}>0\right].
\end{equation}
Hence the distribution of $Y_t$ under ${\mathbf P}_x(\cdot\mid M_{\infty}>0)$ 
is weakly convergent to the average over the Fr\'echet distributions with parameter $\alpha$ 
scaled by $c_*M_{\infty}$ 
(see, e.g., \cite[Theorem 1.12]{B17} and references therein 
for the terminologies about external distributions). 
On the other hand, if $d>\alpha$, 
then ${\mathbf M}$ is transient so that 
${\mathbf P}_x(\{e_0=\infty\}\cap \{M_{\infty}=0\})>0$. 
In particular, we do not know the validity of \eqref{eq:weak}. 

For $R>0$, let $Z_t^R=Z_t(\overline{B(R)}^c)$. 
The next theorem determines the long time asymptotic behavior of the tail distribution 
of $L_t$:  
\begin{thm}\label{thm:tail}
Let $a$ be a positive measurable function on $(0,\infty)$ 
such that $a(t)\rightarrow\infty$ as $t\rightarrow\infty$, 
and let $R(t)=(e^{-\lambda t}a(t))^{1/\alpha}$. 
\begin{enumerate}
\item[{\rm (i)}] 
The next equality holds locally uniformly in $x\in {\mathbb R}^d$.
$$\lim_{t\rightarrow\infty}\frac{{\mathbf P}_x(L_t>R(t))}{{\mathbf E}_x\left[Z_t^{R(t)}\right]}=1.$$
\item[{\rm (ii)}]
For each $k\in {\mathbb N}$, the next equality holds locally uniformly in $x\in {\mathbb R}^d$.
$$\lim_{t\rightarrow\infty}{\mathbf P}_x(Z_t^{R(t)}=k\mid L_t>R(t))=
\begin{cases} 1 & (k=1), \\ 0 & (k\geq 2).\end{cases}$$
\end{enumerate}
\end{thm}

The statement of this theorem is similar to 
those of \cite[Theorem 2.5 and Corollary 2.6]{NS21+}; 
however, the tail distribution of the maximal displacement 
for the branching symmetric stable process is completely different from 
that for the branching Brownian motion (see \cite[(2.16), (2,17)]{NS21+}).   
In fact, combining Theorem \ref{thm:tail} 
with Lemmas \ref{lem:fk-asymp} and \ref{lem:moment}, and \eqref{eq:asymp-r}, we have as $t\rightarrow\infty$,
\begin{equation}\label{eq:tail}
{\mathbf P}_x(L_t>R(t))\sim {\mathbf E}_x\left[Z_t^{R(t)}\right]\sim \frac{c_*h(x)}{a(t)}.
\end{equation}

We omit the proof of Theorem \ref{thm:tail} 
because it is identical with those of \cite[Theorem 2.5 and Corollary 2.6]{NS21+}, respectively.

\subsection{Examples}\label{subsect:exam}
In this subsection, we present three examples 
to which the results in the previous subsection are applicable. 
\begin{exam}\rm 
Let $d=1$ and $\alpha\in(1,2)$. 
Then $\delta_0$, the Dirac measure at the origin, belongs to the Kato class. 
Let $\overline{{\mathbf M}}$ be a branching symmetric $\alpha$-stable process on ${\mathbb R}$
with branching rate $\mu=c \delta_0 \ (c>0)$ 
and branching mechanism ${\mathbf p}=\{p_n(x)\}_{n=0}^{\infty}$. 
We assume that $p_0(0)+p_2(0)=1$ for simplicity. 
Then ${\mathbf P}_x(e_0=\infty)>0$ if and only if $p_2(0)>1/2$ (\cite[Example 4.4]{S08}).
In particular, if $m=2p_2(0)>1$, then 
$$\lambda:=\lambda((Q-1)\mu)
=-\left\{\frac{c(m-1)2^{1/\alpha}}{\alpha\sin(\pi/\alpha)}\right\}^{\alpha/(\alpha-1)}$$
and 
$$
c_*=\frac{C_{1,\alpha}\omega_1}{\alpha(-\lambda)^2}\times c(m-1)\int_{{\mathbb R}}h(y)\,\delta_0({\rm d}y)
=\frac{2c(m-1)C_{1,\alpha}}{\alpha(-\lambda)^2}h(0).
$$
With these $\lambda$ and $c_*$, \eqref{eq:weak} and \eqref{eq:tail} hold.  
\end{exam}

\begin{exam}\rm 
Let $1<\alpha<2$ and $d>\alpha$. 
For $r>0$, let $\delta_r$ be the surface measure on $\partial B(r)=\{y\in {\mathbb R}^d \mid |y|=r\}$. 
Let $\overline{{\mathbf M}}$ be a 
branching symmetric $\alpha$-stable process on ${\mathbb R}^d$  
with branching rate $\mu=c\delta_r \ (c>0)$ and branching mechanism ${\mathbf p}=\{p_n(x)\}_{n=0}^{\infty}$. 
We assume that $p_0\equiv p_0(x)$, $p_2\equiv p_2(x)$ and $p_0+p_2=1$.  
Then ${\mathbf P}_x(e_0=\infty)>0$ holds irrelevantly of the value of $p_2$ 
because ${\mathbf M}$ is transient.
If we let  $m=2p_2$ and  $\lambda:=\lambda((Q-1)\mu)$, then 
$\lambda<0$ if and only if $p_2>1/2$ and 
$$
r>\left\{\frac{\sqrt{\pi}\Gamma((d+\alpha-2)/2)\Gamma(\alpha/2)}
{c(m-1)\Gamma((d-\alpha)/2)\Gamma((\alpha-1)/2)}\right\}^{1/(\alpha-1)}
$$
(see \cite[Example 4.7]{S08} and references therein). 
Under this condition, \eqref{eq:weak-0} and \eqref{eq:tail} hold.
\end{exam}

Assume that $1<\alpha<2$ and $d>\alpha$. 
Let $r>0$ and $\mu_r({\rm d}x)={\bf 1}_{B(r)}(x)\,{\rm d}x$. 
To present the last example, 
we estimate 
$$\check{\lambda}_{\beta}
=\inf\left\{{\cal E}(u,u) \mid u\in {\cal F}, \beta\int_{B(r)}u^2\,{\rm d}x=1\right\} \quad (\beta>0),$$
which is the bottom of the spectrum for the time changed Dirichlet form of $({\cal E},{\cal F})$ 
with respect to the measure $\beta\mu_r$   
(see, e.g., \cite[Section 3]{ST05} for details). 

Let $\check{\lambda}=\check{\lambda}_1$, and let
$v(x)=\int_{B(r)}G(x,y)\,{\rm d}y$ 
be the $0$-potential of the measure $\mu_r$. 
Then 
\begin{equation}\label{eq:upper}
\check{\lambda}\leq \frac{1}{\|v\|_{L^2(B(r))}^2}{\cal E}(v,v)
=\frac{1}{\|v\|_{L^2(B(r))}^2}\int_{B(r)}v\,{\rm d}x\leq \frac{1}{\inf_{y\in B(r)}v(y)}.
\end{equation}
Let 
$$I_{d,\alpha}=\alpha \int_0^1u^{d-1}(1+u)^{\alpha-d}\,{\rm d}u, \quad 
\kappa_{d,\alpha}=\frac{\alpha\Gamma(d/2)\Gamma(\alpha/2)}{2^{2-\alpha}\Gamma((d-\alpha)/2)}.$$
Recall that $\omega_d=2\pi^{d/2}\Gamma(d/2)^{-1}$ is the surface area of the unit ball in ${\mathbb R}^d$. 
Then for any $y\in B(r)$, $|y-z|\leq |y|+|z|\leq r+|z|$ and thus 
\begin{equation*}
\begin{split}
\int_{B(r)}\frac{{\rm d}z}{|y-z|^{d-\alpha}}
\geq \int_{B(r)}\frac{{\rm d}z}{(r+|z|)^{d-\alpha}}
=\omega_d\int_0^r \frac{s^{d-1}}{(r+s)^{d-\alpha}}\,{\rm d}s
=\frac{\omega_dI_{d,\alpha}}{\alpha}r^{\alpha}.
\end{split}
\end{equation*}
Since this inequality and \eqref{eq:green} yield 
$$\inf_{y\in B(r)}v(y)\geq \frac{I_{d,\alpha}r^{\alpha}}{\kappa_{d,\alpha}},$$
we get by \eqref{eq:upper},  
$$
\check{\lambda}\leq \frac{\kappa_{d,\alpha}}{I_{d,\alpha}r^{\alpha}}.
$$
We also know by \cite[Example 3.10]{ST05} that  
$$
\check{\lambda}\geq \frac{\kappa_{d,\alpha}}{r^{\alpha}}.
$$
Noting that $\check{\lambda}_{\beta}=\check{\lambda}/\beta$, we further obtain 
\begin{equation}\label{eq:eigen}
\frac{\kappa_{d,\alpha}}{\beta r^{\alpha}}\leq \check{\lambda}_{\beta}
\leq \frac{\kappa_{d,\alpha}}{\beta I_{d,\alpha}r^{\alpha}}.
\end{equation}
We here note that the lower bound of $\check{\lambda}$ 
in \cite[Example 3.10]{ST05} is incorrect because of the computation error. 

Let $\lambda_{\beta}=\lambda(\beta\mu_r)$. 
Then by \cite[Lemma 2.2]{TT07}, 
$\lambda_{\beta}<0$ if and only if $\check{\lambda}_{\beta}<1$. 
Hence by \eqref{eq:eigen}, we obtain
\begin{equation}\label{eq:eigen-1}
r>\left(\frac{\kappa_{d,\alpha}}{\beta I_{d,\alpha}}\right)^{1/\alpha} \Rightarrow \lambda_{\beta}<0, \quad 
r\leq \left(\frac{\kappa_{d,\alpha}}{\beta}\right)^{1/\alpha} \Rightarrow \lambda_{\beta}\geq 0.
\end{equation}
We do not know if  $\lambda_{\beta}$ is negative or not 
for $(\kappa_{d,\alpha}/\beta)^{1/\alpha}<r
\leq \left\{\kappa_{d,\alpha}/(\beta I_{d,\alpha})\right\}^{1/\alpha}$.

\begin{exam}\rm 
Let $1<\alpha<2$ and $d>\alpha$. 
For $r>0$, let $\mu_r({\rm d}x)={\bf 1}_{B(r)}(x)\,{\rm d}x$. 
Let $\overline{{\mathbf M}}$ be a branching symmetric $\alpha$-stable process on ${\mathbb R}^d$
with branching rate $\mu=c\mu_r \ (c>0)$ and branching mechanism ${\mathbf p}=\{p_n(x)\}_{n=0}^{\infty}$. 
We assume that $p_0\equiv p_0(x)$, $p_2\equiv p_2(x)$ and $p_0+p_2=1$.  
Then ${\mathbf P}_x(e_0=\infty)>0$ holds irrelevantly of the value of $p_2$.

Let  $m=2p_2$ and  $\lambda:=\lambda((Q-1)\mu)$.  
Assume that $p_2>1/2$. 
Then by  \eqref{eq:eigen-1}, we have the following: 
if  
$r>\left\{\kappa_{d,\alpha}/(c(m-1)I_{d,\alpha})\right\}^{1/\alpha}$, 
then $\lambda<0$ so that \eqref{eq:weak-0} and \eqref{eq:tail} hold. 
On the other hand, if $r\leq \{\kappa_{d,\alpha}/(c(m-1))\}^{1/\alpha}$, 
then we have $\lambda=0$ so that Assumption \ref{assum:weak} fails. 
\end{exam}

\section{Proof of Theorem \ref{thm:weak}}
\label{sect:proof}

Once we obtain the asymptotic behaviors of the Feynman-Kac functionals as 
in Subsection \ref{subsect:asymp-fk}, 
we can establish Theorem \ref{thm:weak} along the way as for the proof of \cite[Theorem 2.4]{NS21+}.

Let $\overline{{\mathbf M}}=(\{{\mathbf X}_t\}_{t\geq 0},\{{\mathbf P}_{\mathbf x}\}_{{\mathbf x}\in {\mathbf X}})$ 
be a branching symmetric $\alpha$-stable process on ${\mathbb R}^d$ 
with branching rate $\mu$ and branching mechanism ${\mathbf p}$ 
so that  Assumption \ref{assum:weak} is fulfilled. 
Let $\nu=(Q-1)\mu$ and $\lambda=\lambda((Q-1)\mu)$, 
and let $c_*$ be the corresponding value in \eqref{eq:c-ast}.  
Recall that for $\kappa>0$, $R^{\kappa}(t)=(e^{-\lambda t}\kappa)^{1/\alpha}$. 
\begin{lem}\label{lem:uniform-1}
Let $K$ be a compact set in ${\mathbb R}^d$. 
Then 
$$
\lim_{\kappa\rightarrow\infty}\limsup_{t\rightarrow\infty}
\sup_{x\in K}
\left|\frac{\kappa}{h(x)}{\mathbf P}_x(L_t>R^{\kappa}(t))-c_*\right|=0
$$
and for any $c>0$,
$$\lim_{\gamma \rightarrow +0}\limsup_{t\rightarrow\infty}
\sup_{x\in K}
\left|\frac{1}{\gamma h(x)}{\mathbf E}_x\left[1-e^{-\gamma c M_t}\right]-c\right|=0.$$
\end{lem}

We omit the proof of Lemma \ref{lem:uniform-1}; 
by using Lemma \ref{lem:second-1}, 
we can show Lemma \ref{lem:uniform-1} in the same way as for \cite[Lemma 4.1]{NS21+}.

Let ${\cal L}$ be the totality of compact sets in ${\mathbb R}^d$. 
\begin{lem}\label{lem:uniform-2}
The following equalities hold{\rm :}
\begin{equation}\label{eq:uni-2}
\lim_{\kappa\rightarrow\infty}\sup_{L\in {\cal L}}\limsup_{t\rightarrow\infty}
\sup_{x\in L}
\left|\frac{\kappa}{h(x)}{\mathbf P}_x(L_t>R^{\kappa}(t))-c_*\right|=0
\end{equation}
and for any $c>0$,
$$\lim_{\gamma \rightarrow +0}\sup_{L\in {\cal L}}\limsup_{t\rightarrow\infty}
\sup_{x\in L}
\left|\frac{1}{\gamma h(x)}{\mathbf E}_x\left[1-e^{-\gamma c M_t}\right]-c\right|=0.$$
\end{lem}

\begin{proof}
We can prove the assertion in the same way as for \cite[Proposition 4.2]{NS21+}. 
We here prove \eqref{eq:uni-2} only. 
By the Chebyshev inequality and Lemma \ref{lem:moment},
\begin{equation*}
\begin{split}
{\mathbf P}_x(L_t>R^{\kappa}(t))
&={\mathbf P}_x\left(Z_t^{R^{\kappa}(t)}\geq 1\right)\\
&\leq {\mathbf E}_x\left[Z_t^{R^{\kappa}(t)}\right]
=E_x\left[e^{A_t^{(Q-1)\mu}};|X_t|>R^{\kappa}(t)\right].
\end{split}
\end{equation*}
Hence by Lemma \ref{lem:second-1}, 
$$
\limsup_{\kappa\rightarrow\infty}\sup_{L\in {\cal L}}\limsup_{t\rightarrow\infty}
\sup_{x\in L}
\frac{\kappa}{h(x)}{\mathbf P}_x(L_t>R^{\kappa}(t))\leq c_*.
$$
Then for the proof of  \eqref{eq:uni-2}, it suffices to show that 
\begin{equation}\label{eq:inf-lower}
\liminf_{\kappa\rightarrow\infty}\inf_{L\in {\cal L}}\liminf_{t\rightarrow\infty}\inf_{x\in L}
\frac{\kappa}{h(x)}{\mathbf P}_x(L_t>R^{\kappa}(t))\geq c_*.
\end{equation}

In what follows, we give a proof of \eqref{eq:inf-lower}. 
Lemma \ref{lem:uniform-1} says that 
for any $\varepsilon>0$ and $K\in {\cal L}$, 
there exists $\kappa_0=\kappa_0(\varepsilon,K)>0$ such that 
for any $\kappa\geq \kappa_0$, 
there exists $T_0=T_0(\varepsilon, K,\kappa)>0$ such that 
for any $t\geq T_0$ and $x\in K$,
$$|\kappa{\mathbf P}_x(L_t>R^{\kappa}(t))-c_*h(x)|<\varepsilon h(x).$$

Let $t>0$ and $0\leq s<t$. Then 
$$
R^{\kappa}(t)
=(e^{-\lambda(t-s)})^{1/\alpha}(\kappa e^{-\lambda s})^{1/\alpha}
=R^{\kappa e^{-\lambda s}}(t-s)
$$
and $\kappa e^{-\lambda s}\geq \kappa\geq \kappa_0$. 
Hence for any $T>0$, $t\geq T+T_0$ and $s\in [0,T]$, 
\begin{equation}\label{eq:bound-l}
|\kappa{\mathbf P}_x(L_{t-s}>R^{\kappa}(t))-c_*h(x)|<\varepsilon h(x), \ x\in K.
\end{equation}

Fix $K\in {\cal L}$ which includes the support of $\mu$. 
Let $\sigma$ be the first hitting time to $K$ 
of some particle. 
Then $\sigma$ is relevant to the initial particle only 
because particles can not branch outside $K$. 
We use the same notation $\sigma$ to denote 
the first hitting time to $K$ 
of a symmetric $\alpha$-stable process 
${\mathbf M}=(\{X_t\}_{t\geq 0},\{P_x\}_{x\in {\mathbb R}^d})$.   
Since $X_{\sigma}\in K$, 
there exists $\kappa_1=\kappa_1(\varepsilon)>0$ for any $\varepsilon>0$  
such that for any $\kappa\geq\kappa_1$, 
there exists $T_1=T_1(\varepsilon,\kappa)>0$ such that 
for any $T>0$ and $t\geq T+T_1$, 
we have by the strong Markov property and \eqref{eq:bound-l},
\begin{equation}\label{eq:est-h}
\begin{split}
\kappa{\mathbf P}_x(L_t>R^{\kappa}(t))
&\geq E_x\left[\kappa{\mathbf P}_{X_{\sigma}}(L_{t-s}>R^{\kappa}(t))|_{s=\sigma};\sigma \leq T\right]\\
&\geq (c_*-\varepsilon)E_x\left[h(X_{\sigma});\sigma\leq T\right], \ x\in {\mathbb R}^d.
\end{split}
\end{equation}

Since $e^{\lambda t+A_t^{(Q-1)\mu}}h(X_t)$ is a $P_x$-martingale 
and $P_x(A_{t\wedge \sigma}^{(Q-1)\mu}=0)=1$ for any $t\geq 0$,  
the optional stopping theorem yields 
$$E_x\left[e^{\lambda(T\wedge \sigma)}h(X_{T\wedge \sigma})\right]
=E_x\left[e^{\lambda(T\wedge \sigma)+A_{T\wedge \sigma}^{(Q-1)\mu}}h(X_{T\wedge \sigma})\right]
=h(x)$$
and thus 
\begin{equation*}
\begin{split}
E_x\left[h(X_{\sigma});\sigma\leq T\right]
&\geq E_x\left[e^{\lambda\sigma}h(X_{\sigma});\sigma\leq T\right]\\
&=E_x\left[e^{\lambda(T\wedge \sigma)}h(X_{T\wedge \sigma})\right]
-E_x\left[e^{\lambda T}h(X_T);T<\sigma\right]\\
&\geq h(x)-e^{\lambda T}\|h\|_{\infty}.
\end{split}
\end{equation*}
Then by \eqref{eq:est-h}, we have for any $t\geq T+T_1$,
$$\kappa {\mathbf P}_x(L_t>R^{\kappa}(t))
\geq (c_*-\varepsilon)(h(x)-e^{\lambda T}\|h\|_{\infty}), \ x\in {\mathbb R}^d.$$
In particular, for any $L\in {\cal L}$ and $t\geq T+T_1$,
$$\inf_{x\in L}\frac{\kappa}{h(x)}{\mathbf P}_x(L_t>R^{\kappa}(t))
\geq (c_*-\varepsilon)\left(1-\frac{e^{\lambda T}\|h\|_{\infty}}{\inf_{x\in L}h(x)}\right).$$
Letting $t\rightarrow\infty$ and then $T\rightarrow\infty$, we have 
$$
\liminf_{t\rightarrow\infty}\inf_{x\in L}\frac{\kappa}{h(x)}{\mathbf P}_x(L_t>R^{\kappa}(t))
\geq c_*-\varepsilon.
$$
Furthermore, since the right hand side above is independent of the choice of $L\in {\cal L}$, 
we obtain \eqref{eq:inf-lower} by letting $\kappa\rightarrow\infty$ and then $\varepsilon\rightarrow 0$. 
\qed
\end{proof}

\noindent
{\it Proof of Theorem {\rm \ref{thm:weak}}.}
We follow the argument of \cite[Proof of Theorem 2.4]{NS21+}. 
In what follows, we write $R(t)=R^{\kappa}(t)$ for simplicity. 
Since ${\mathbf P}_x(L_t<\infty)=1$ for any $t\geq 0$,  
there exists $r_1=r_1(\varepsilon,T_1)$ for any $\varepsilon>0$ and $T_1>0$ 
such that ${\mathbf P}_x(L_{T_1}>r_1)\leq \varepsilon$. 
Hence for any $t\geq T_1$, 
$${\mathbf P}_x(L_t\leq R(t))\leq {\mathbf P}_x(L_t\leq R(t), L_{T_1}\leq r_1)+\varepsilon,
$$ 
which yields 
\begin{equation}\label{eq:lower}
\begin{split}
&{\mathbf P}_x(L_t>R(t))-{\mathbf E}_x\left[1-\exp\left(-\kappa^{-1}c_*M_t\right)\right]\\
&\geq {\mathbf E}_x\left[\exp\left(-\kappa^{-1}c_*M_t\right)-{\bf 1}_{\{L_t\leq R(t)\}}; L_{T_1}\leq r_1\right]
-\varepsilon.
\end{split}
\end{equation}

Recall that $e_0=\inf\{t>0 \mid Z_t=0\}$ is the extinction time of $\overline{\mathbf M}$. 
Then for any $t\geq e_0$, $M_t=0$ and $L_t=0$ by definition. 
Therefore, by the Markov property,  
\begin{equation*}
\begin{split}
&{\mathbf E}_x\left[\exp\left(-\kappa^{-1}c_*M_t\right)-{\bf 1}_{\{L_t\leq R(t)\}}; L_{T_1}\leq r_1\right]\\
&={\mathbf E}_x\Bigl[
\left\{{\mathbf E}_{{\mathbf X}_{T_1}}\left[
\exp\left(-\kappa^{-1}c_*e^{\lambda T_1}M_{t-T_1}\right)\right]
-{\mathbf P}_{{\mathbf X}_{T_1}}(L_{t-T_1}\leq R(t))\right\}; \\ 
&\qquad \qquad T_1<e_0, L_{T_1}\leq r_1\Bigr]\\
&={\rm (VI)}.
\end{split}
\end{equation*}

By Lemma \ref{lem:uniform-2}, 
there exists $\kappa_0=\kappa_0(\delta)>0$ for any $\delta\in(0,c_*)$ 
such that if $T>0$ satisfies $\kappa e^{-\lambda T}\geq \kappa_0$, then 
$$\sup_{L\in {\cal L}}\limsup_{t\rightarrow\infty}\sup_{x\in L}
\left|\frac{\kappa e^{-\lambda T}}{h(x)}{\mathbf P}_x(L_{t-T}>R(t))-c_*\right|<\delta.$$
Let $T_1=T_1(\kappa_0)$ satisfy $\kappa e^{-\lambda T_1}\geq \kappa_0$. 
Then there exists $T_2=T_2(\varepsilon,\delta,T_1)>0$ such that 
for any $y\in \overline{B(r_1)}$ and $t\geq T_1+T_2$, 
\begin{equation}\label{eq:tail-1}
\kappa^{-1}(c_*-\delta)e^{\lambda T_1}h(y)
\leq {\mathbf P}_y(L_{t-T_1}>R(t))
\leq \kappa^{-1}(c_*+\delta)e^{\lambda T_1}h(y).
\end{equation}
Note that 
$1-x\leq e^{-x}$ for any $x\in {\mathbb R}$ 
and there exists $r_0(\delta)>0$ for any $\delta>0$ 
such that $1-x\geq e^{-(1+\delta)x}$ for any $x\in [0,r_0(\delta)]$.
Hence if we take $T_1$ so large that 
$\kappa^{-1}(c_*+\delta)e^{\lambda T_1}\|h\|_{\infty}\leq r_0(\delta)$, 
then by \eqref{eq:tail-1},
\begin{equation*}
\begin{split}
\exp\left(-(1+\delta)\kappa^{-1}(c_*+\delta)e^{\lambda T_1}h(y)\right)
&\leq {\mathbf P}_y(L_{t-T_1}\leq R(t))\\
&\leq \exp\left(-\kappa^{-1}(c_*-\delta)e^{\lambda T_1}h(y)\right).
\end{split}
\end{equation*}
By Lemma \ref{lem:uniform-2}, we also obtain  
\begin{equation*}
\begin{split}
\exp\left(-(1+\delta)\kappa^{-1}(c_*+\delta)e^{\lambda T_1}h(y)\right)
&\leq 
{\mathbf E}_y\left[\exp\left(-\kappa^{-1}c_*e^{\lambda T_1}M_{t-T_1}\right)\right]\\
&\leq \exp\left(-\kappa^{-1}(c_*-\delta)e^{\lambda T_1}h(y)\right)
\end{split}
\end{equation*}
so that for any $t\geq T_1+T_2$,
$$
{\rm (VI)}
\geq {\mathbf E}_x\left[\exp\left(-(1+\delta)\kappa^{-1}(c_*+\delta)M_{T_1}\right)\right]
-{\mathbf E}_x\left[\exp\left(-\kappa^{-1}(c_*-\delta)M_{T_1}\right)\right].
$$
Since the right hand side goes to $0$ as 
$t\rightarrow\infty$, $T_1\rightarrow\infty$ and  $\delta\rightarrow+0$,
we have by \eqref{eq:lower},
$$
\liminf_{t\rightarrow\infty}
\left({\mathbf P}_x(L_t>R(t))-{\mathbf E}_x\left[1-\exp\left(-\kappa^{-1}c_*M_t\right)\right]\right)
\geq 0.
$$
In the same way, we also have 
$$
\limsup_{t\rightarrow\infty}
\left({\mathbf P}_x(L_t>R(t))-{\mathbf E}_x\left[1-\exp\left(-\kappa^{-1}c_*M_t\right)\right]\right)
\leq 0
$$
so that the proof is complete. 
\qed

\ack
The author would like to thank Professor Yasuhito Nishimori and 
Professor Masayoshi Takeda 
for valuable comments on the draft of this paper. 
He is also grateful to the referee 
for his/her careful reading of the manuscript and valuable comments.


\begin{thebibliography}{99}

\bibitem{ABM91}
S.\ Albeverio, P.\ Blanchard and Z.\ M.\ Ma:
Feynman-Kac semigroups in terms of signed smooth measures,
in ``Random Partial Differential Equations'' (U. Hornung et al.\, Eds.),
Birkh\"auser, Basel, 1991, pp.\ 1--31.




\bibitem{Be04} 
A.\ Ben Amor: 
Invariance of essential spectra for generalized Schr\"odinger operators, 
Math.\ Phys.\ Electron.\ J. {\bf 10} (2004), Paper 7, 18 pp.




\bibitem{BHR17}
A.\ Bhattacharya, R.\ S.\ Hazra and P.\ Roy: 
Point process convergence for branching random walks with regularly varying steps, 
Ann.\ Inst.\ Henri Poincar\'e Probab.\ Stat. {\bf 53} (2017), 802--818. 

\bibitem{BG60}
R.\ M.\ Blumenthal and R.\ K.\ Getoor: Some theorems on stable processes, 
Trans.\ Amer.\ Math.\ Soc. {\bf 95} (1960), 263--273. 

\bibitem{Bo20} 
S.\ Bocharov: 
Limiting distribution of particles near the frontier in the catalytic branching Brownian motion, 
Acta Appl.\ Math. {\bf 169} (2020), 433--453.


\bibitem{BH14} 
S.\ Bocharov and S.\ C.\ Harris: 
Branching Brownian motion with catalytic branching at the origin, 
Acta Appl.\ Math. {\bf 134} (2014), 201--228.


\bibitem{BH16}
S.\ Bocharov and S.\ C.\ Harris: 
Limiting distribution of the rightmost particle in catalytic branching Brownian motion,
Electron.\ Commun.\ Probab. {\bf 21} (2016), 12pp.

\bibitem{B17}
A.\ Bovier: 
Gaussian Processes on Trees. From Spin Glasses to Branching Brownian Motion, 
Cambridge University Press, Cambridge, 2017.



\bibitem{B78}
M.\ D.\ Bramson:
Maximal displacement of branching Brownian motion,
Comm.\ Pure Appl.\ Math. {\bf 31} (1978), 531--581.




\bibitem{BEKS94}
J.\ F.\ Brasche, P.\ Exner, Yu.\ A.\ Kuperin and  P.\ \u{S}eba:
Schr\"odinger operators with singular interactions, 
J. Math.\ Anal.\ Appl. {\bf 184} (1994), 112--139.

\bibitem{B18}
E.\ Vl.\ Bulinskaya:
Spread of a catalytic branching random walk on a multidimensional lattice, 
Stochastic Process.\ Appl. {\bf 28} (2018), 2325--2340.


\bibitem{B20}
E.\ Vl.\ Bulinskaya: 
Fluctuations of the propagation front of a catalytic branching walk, 
Theory Probab.\ Appl. {\bf 64} (2020), 513--534. 

\bibitem{B21}
E.\ VI.\ Bulinskaya: 
Maximum of catalytic branching random walk with regularly varying tails, 
J. Theoret.\ Probab. {\bf 34} (2021), 141–-161. 



\bibitem{CH14}
P.\ Carmona and  Y.\ Hu:
The spread of a catalytic branching random walk,
Ann.\ Inst.\ H. Poincar\'e Probab.\ Statist. {\bf 50} (2014), 327--351.

\bibitem{C02}
Z.-Q.\ Chen:
Gaugeability and conditional gaugeability, 
Trans.\ Amer.\ Math.\ Soc. {\bf 354} (2002), 4639--4679.


\bibitem{CF12}
Z.-Q.\ Chen and M.\ Fukushima:
Symmetric Markov Processes, Time Change, and Boundary Theory, 
Pinceton University Press, Princeton and Oxford, 2012. 


\bibitem{D83}
R.\ Durrett: 
Maxima of branching random walks, 
Z. Wahrsch.\ Verw.\ Gebiete {\bf 62} (1983), 165--170. 

\bibitem{E84}
K.\ B.\ Erickson: 
Rate of expansion of an inhomogeneous branching process of Brownian particles,
Z. Wahrsch.\ Verw.\ Gebiete {\bf 66} (1984), 129--140. 


\bibitem{FOT11} 
M.\ Fukushima, Y.\ Oshima and M.\ Takeda: 
Dirichlet Forms and Symmetric Markov Processes, 2nd rev.\ and ext.\ ed., 
Walter de Gruyter, 2011.


\bibitem{GMV17}
A.\ Getan, S.\ Molchanov and B.\ Vainberg:
Intermittency for branching walks with heavy tails, 
Stoch.\ Dyn. {\bf 17} (2017), 1750044, 14 pp. 



\bibitem{HS78}
P.\ R.\ Halmos and V.\ S.\ Sunder: 
Bounded Integral Operators on $L^2$ Spaces, 
Springer-Verlag, Berlin Heidelberg, 1978. 



\bibitem{KLZ21+}
Y.\ H.\ Kim, E.\ Lubetzky and O.\ Zeitouni:
The maximum of branching Brownian motion in ${\mathbb R}^d$, 
preprint, available at arXiv:2104.07698.


\bibitem{LS88}
S.\ Lalley and T.\ Sellke:
Traveling waves in inhomogeneous branching Brownian motions. I, 
Ann.\ Probab. {\bf 16} (1988), 1051--1062. 

\bibitem{LS16}
S.\ P.\ Lalley and Y.\ Shao:
Maximal displacement of critical branching symmetric stable processes, 
Ann.\ Inst.\ Henri Poincar\'e Probab.\ Stat. {\bf 52} (2016), 1161--1177. 

\bibitem{N21+}
Y.\ Nishimori:
Limiting distributions for particles near the frontier of spatially inhomogeneous branching Brownian motions,
preprint, available at arXiv:2104.13063.



\bibitem{NS21+}
Y.\ Nishimori and Y.\ Shiozawa: 
Limiting distributions for the maximal displacement of branching Brownian motions, 
to appear in J. Math.\ Soc.\ Japan,  DOI:10.2969/jmsj/85158515





\bibitem{RS4} 
M.\ Reed and B.\ Simon:
Methods of Modern Mathematical Physics. IV. Analysis of Operators, 
Academic Press, New York-London, 1978. 

\bibitem{S08}  
Y.\ Shiozawa:
Exponential growth of the numbers of particles for branching symmetric $\alpha$-stable processes, 
J. Math.\ Soc.\ Japan {\bf 60} (2008), 75--116.

\bibitem{S18}
Y.\ Shiozawa:
Spread rate of branching Brownian motions, 
Acta Appl.\ Math. {\bf 155} (2018), 113--150.


\bibitem{S19}
Y.\ Shiozawa: 
Maximal displacement and population growth for branching Brownian motions, 
Illinois J. Math. {\bf 63} (2019), 353--402.


\bibitem{ST05}
Y.\ Shiozawa and M.\ Takeda:
Variational formula for Dirichlet forms 
and estimates of principal eigenvalues for symmetric $\alpha$-stable processes, 
Potential Anal. {\bf 23} (2005), 135--151.

\bibitem{SBM21+}
R.\ Stasi\'nski, J.\ Berestycki and B.\ Mallein:
Derivative martingale of the branching Brownian motion in dimension $d\geq 1$, 
Ann.\ Inst.\ Henri Poincar\'e Probab.\ Stat. {\bf 57} (2021), 1786--1810.


\bibitem{SV96} 
P.\ Stollmann and J.\ Voigt: 
Perturbation of Dirichlet forms by measures,
Potential Anal. {\bf 5} (1996), 109--138. 



\bibitem{T02}
M.\ Takeda: 
Conditional gaugeability and subcriticality of generalized Schr\"odinger operators,
J. Funct.\ Anal. {\bf 191} (2002), 343--376.


\bibitem{T08} 
M.\ Takeda: 
Large deviations for additive functionals of symmetric stable processes, 
J. Theoret.\ Probab. {\bf 21} (2008), 336--355.

\bibitem{TT07}
M.\ Takeda and K.\ Tsuchida:
Differentiability of spectral functions for symmetric $\alpha$-stable processes processes, 
Trans.\ Amer.\ Math.\ Soc. {\bf 359} (2007), 4031--4054.


\bibitem{W17} 
M.\ Wada:
Asymptotic expansion of resolvent kernels and behavior of spectral functions 
for symmetric stable processes,
J. Math.\ Soc.\ Japan {\bf 69} (2017), 673--692. 
\end{thebibliography}
\end{document}